\documentclass[psamsfonts,11pt]{amsart}

\usepackage{amssymb,amsfonts}
\usepackage[all,arc]{xy}
\usepackage{enumerate}
\usepackage{mathrsfs}
\usepackage{url}
\usepackage{physics}
\usepackage[margin=1in]{geometry}
\usepackage{makecell}
\usepackage{tikz}
\usepackage{forest}
\usepackage{hyperref}
\usepackage{graphics}

\newtheorem{thm}{Theorem}[section]
\newtheorem{cor}[thm]{Corollary}
\newtheorem{prop}[thm]{Proposition}
\newtheorem{lem}[thm]{Lemma}
\newtheorem{conj}[thm]{Conjecture}

\theoremstyle{definition}
\newtheorem{defn}[thm]{Definition}

\newtheorem{exmp}[thm]{Example}

\theoremstyle{remark}

\makeatletter
\let\c@equation\c@thm
\makeatother
\numberwithin{equation}{section}

\title{Wilf Equivalences for Patterns in Rooted Labeled Forests}

\author{Michael Ren}

\begin{document}

\begin{abstract}

Building off recent work of Garg and Peng, we continue the investigation into classical and consecutive pattern avoidance in rooted forests, resolving some of their conjectures and questions and proving generalizations whenever possible. Through extensions of the forest Simion-Schmidt bijection introduced by Anders and Archer, we demonstrate a new family of forest-Wilf equivalences, completing the classification of forest-Wilf equivalence classes for sets consisting of a pattern of length 3 and a pattern of length at most $5$. We also find a new family of nontrivial c-forest-Wilf equivalences between single patterns using the forest analogue of the Goulden-Jackson cluster method, showing that a $(1-o(1))^n$-fraction of patterns of length $n$ satisfy a nontrivial c-forest-Wilf equivalence and that there are c-forest-Wilf equivalence classes of patterns of length $n$ of exponential size. Additionally, we consider a forest analogue of super-strong-c-Wilf equivalence, introduced for permutations by Dwyer and Elizalde, showing that super-strong-c-forest-Wilf equivalences are trivial by enumerating linear extensions of forest cluster posets.

\end{abstract}

\maketitle

\section{Introduction}
\label{intro}

A sequence of distinct integers is said to \emph{avoid} a permutation, or pattern, $\pi=\pi(1)\cdots\pi(k)$ of $[k]=\{1,\ldots,k\}$ if it contains no subsequence that is in the same relative order as $\pi$. The study of pattern avoidance in permutations of $[n]$ was started in 1968 by Knuth in \cite{K}, where stack sorting was linked to permutations avoiding the pattern $231$. Since then, pattern avoidance has blossomed into a very active area of research, with many connections made to classical and contemporary results in enumerative and algebraic combinatorics \cite{Ki}. Different variants of permutation pattern avoidance, for example avoidance of consecutive patterns \cite{E} and of generalized patterns \cite{S}, have also been extensively studied, along with notions of pattern avoidance in other combinatorial structures such as binary trees \cite{R}, posets \cite{HW}, heaps \cite{LPRS}, and forests \cite{BLNPPRT}.

The variant of pattern avoidance that we investigate is in rooted labeled forests, a notion introduced in 2018 by Anders and Archer in \cite{AA}. Here, we consider unordered (non-planar) rooted forests on $n$ vertices such that each vertex has a different label in $[n]$, which we call \emph{rooted forests} on $[n]$. Such a forest is then said to \emph{avoid} a pattern $\pi$ if the sequence of labels from the root to any leaf avoids $\pi$ in the sense described in the previous paragraph. As a special case, this includes the aforementioned case of permutation pattern avoidance when the forest is taken to be a path. Anders and Archer find the number $f_n(S)$ of forests avoiding a set $S$ of patterns in \cite{AA} for certain sets $S$. They also study \emph{forest-Wilf equivalence}, the phenomenon when $f_n(S)=f_n(S')$ for different sets $S$ and $S'$ of patterns and all $n\in\mathbb N$. Their work was continued by Garg and Peng in \cite{GP} where the authors posed several open questions, some of which we resolve in this paper.

A classical result in enumerating pattern-avoiding permutations is that for any pattern $\pi$ of length $3$, the number $|\text{Av}_\pi(n)|$ of permutations of $[n]$ avoiding $\pi$ is the $n$th Catalan number $C_n=\frac1{n+1}\binom{2n}n$, shown for example through the \emph{Simion-Schmidt bijection} in \cite{SS}. The independence of this count from the specific pattern being avoided launched the study of \emph{Wilf equivalence}, the phenomenon when $|\text{Av}_S(n)|=|\text{Av}_{S'}(n)|$ for different sets $S$ and $S'$ of patterns. Work of Stankova and West \cite{SW} along with generalizations of the Simion-Schmidt bijection by Backelin, West, and Xin \cite{BWX} show the existence of infinite families of Wilf equivalences that explain nearly all known Wilf equivalences between single patterns. Using a forest variant of the Simion-Schmidt bijection, Anders and Archer proved the forest-Wilf equivalence of $123$ and $132$ in \cite{AA}, which was generalized by Garg and Peng in \cite{GP} through the following theorem.

\begin{thm}[{\cite[Theorem 1.3]{GP}}]
\label{gptwist}
Suppose that the patterns $\pi_1,\ldots,\pi_m$ satisfy the property that $\pi_i(k_i-1)=k_i-1$ and $\pi_i(k_i)=k_i$ where $\pi_i=\pi_i(1)\cdots\pi_i(k_i)$. Then $\{\pi_1,\ldots,\pi_m\}$ and $\{\widetilde\pi_1,\ldots,\widetilde\pi_m\}$ are forest-Wilf equivalent.
\end{thm}

Here, we denote by $\widetilde\pi$ the \emph{twist} of the pattern $\pi=\pi(1)\cdots\pi(k)$, which we define to be the resulting pattern from switching the last two elements so that $\widetilde\pi=\pi(1)\cdots\pi(k-2)\pi(k)\pi(k-1)$. Through a computer calculation done by Garg and Peng, it can be verified that Theorem \ref{gptwist} explains all nontrivial single pattern forest-Wilf equivalences between patterns of length at most $5$ \cite{GPp}. Here, a \emph{trivial} forest-Wilf equivalence is one between the sets $S=\{\pi_1,\ldots,\pi_k\}$ and $\overline S=\{\overline\pi_1,\ldots,\overline\pi_k\}$ where $\overline\pi=k+1-\pi(1),\ldots,k+1-\pi(k)$ is the \emph{complement} of $\pi=\pi(1)\cdots\pi(k)$. Such an equivalence is trivial because inverting the order of the labels bijects forests avoiding $S$ with forests avoiding $\overline S$. However, Theorem \ref{gptwist} does not explain the following forest-Wilf equivalences conjectured by Garg and Peng in \cite{GP}.

\begin{conj}[{\cite[Conjecture 7.1]{GP}}]
\label{gp71}
The following pairs of sets are forest-Wilf equivalent:
\begin{enumerate}
    \item [(i)] $\{123,2413\}$ and $\{132,2314\}$;
    \item [(ii)] $\{123,3142\}$ and $\{132,3124\}$;
    \item [(iii)]
    $\{213,4123\}$ and $\{213,4132\}$.
\end{enumerate}
\end{conj}

By restricting the forest Simion-Schmidt bijection, we prove parts (i) and (ii) of this conjecture in Section \ref{clas}. We also prove the following generalization of part (iii) by recursively applying the forest Simion-Schmidt bijection. Along with Theorem \ref{gptwist} and parts (i) and (ii) of Conjecture \ref{gp71}, Theorem \ref{213twist} explains all nontrivial forest-Wilf equivalences between sets consisting of a pattern of length $3$ and a pattern of length $k\le5$, again by a computer calculation of Garg and Peng \cite{GPp}.

\begin{thm}
\label{213twist}
If $\pi(k)=\pi(k-1)+1=\pi(k-2)+2$ in the pattern $\pi=\pi(1)\cdots\pi(k)$, then $\{213,\pi\}$ and $\{213,\widetilde\pi\}$ are forest-Wilf equivalent.
\end{thm}

Garg and Peng also considered \emph{consecutive} pattern avoidance in forests. A sequence of distinct integers \emph{consecutively avoids} a pattern $\pi$ if it contains no consecutive subsequence that is in the same relative order as $\pi$. The analogues of Wilf equivalence and forest-Wilf equivalence, which we call \emph{c-Wilf equivalence} and \emph{c-forest-Wilf equivalence}, are defined in the same way using the consecutive notion of pattern avoidance instead.

The systematic study of consecutive pattern avoidance in permutations was started in 2003 by Elizalde and Noy in \cite{EN03}, where the number of permutations of $[n]$ avoiding $\pi$ was enumerated for short patterns $\pi$ using analytic techniques. Elizalde and Noy further developed these techniques in \cite{EN}, where they applied the Goulden-Jackson cluster method from \cite{GJ} to asymptotically enumerate more classes of permutations that avoid consecutive patterns. Garg and Peng adapted the cluster method to forests that avoid consecutive patterns in \cite{GP}, proving the following intriguing result.

\begin{thm}[{\cite[Theorem 1.4]{GP}}]
\label{gpcequiv}
The patterns $1324$ and $1423$ are c-forest-Wilf equivalent but not c-Wilf equivalent.
\end{thm}

In contrast, for the classical case all known forest-Wilf equivalences are also Wilf equivalences, and Garg and Peng even conjectured in \cite[Conjecture 7.3]{GP} that forest-Wilf equivalence implies Wilf equivalence. Garg and Peng also asked whether there exist more nontrivial single pattern c-forest-Wilf equivalences after checking all patterns of length at most $5$ with a computer and finding no more \cite{GPp}. We answer their question in the affirmative by exhibiting a rich infinite family of c-forest-Wilf equivalences. Not only are there more nontrivial c-forest-Wilf equivalences, but a relatively large proportion of patterns satisfy such an equivalence. Furthermore, there are exponentially large c-forest-Wilf equivalence classes. In the process of our proof, we somewhat elucidate what properties of forests allow for such c-forest-Wilf equivalences between patterns that are not c-Wilf equivalent as in Theorem \ref{gpcequiv}. Our construction is rather long and is detailed in Section \ref{cons}, so for now we summarize its key properties in the following theorems.

\begin{thm}
\label{cclass}
For all $n$, there exists a c-forest-Wilf equivalence class containing at least $2^{n-4}$ patterns of length $n$.
\end{thm}

\begin{thm}
\label{cequiv}
For any constant $c<1$, there exists an integer $n_c$ such that for all $n>n_c$, at least $c^nn!$ patterns of length $n$ satisfy a nontrivial c-forest-Wilf equivalence.
\end{thm}

In 2018, Dwyer and Elizalde introduced in \cite{DE} the notion of \emph{super-strong c-Wilf equivalence}, refining c-Wilf equivalence for permutations. Roughly speaking, two patterns $\pi$ and $\pi'$ of length $k$ are said to be super-strongly c-Wilf equivalent if for all $n$ and $S\subseteq[n-k+1]$, the number of permutations of $[n]$ whose occurrences of $\pi$ are exactly indexed by $S$ is equal to the number of permutations of $[n]$ whose occurrences of $\pi'$ are exactly indexed by $S$. For example, in the permutation $124365$, the occurrences of the pattern $132$ are indexed by the set $\{2,4\}$, corresponding to the substrings $243$ and $365$ that begin at indices $2$ and $4$. Dwyer and Elizalde show that among a special class of permutations, the \emph{nonoverlapping permutations}, c-Wilf equivalence is equivalent to super-strong c-Wilf equivalence. By the work of Lee and Sah in 2018 in \cite{LS}, c-Wilf equivalence among nonoverlapping patterns is completely understood and determined by the first and last element. Bóna has shown that an asymptotically constant proportion of permutations of $[n]$ are nonoverlapping in \cite{B}, yielding a large number of nontrivial super-strong c-Wilf equivalences. By adapting the techniques for enumerating linear extensions of cluster posets from \cite{DE, LS}, we show that the complete opposite is true for the forest analogue.

\begin{thm}
\label{fssequ}
Any super-strong c-forest-Wilf equivalence between two patterns is trivial.
\end{thm}

Garg and Peng also conjectured that an analogue of Stanley-Wilf limits exist for sets $S$ of patterns in forests. We prove this conjecture for a wide class of sets $S$ which includes all singleton sets in a companion paper \cite{Re}.

The rest of the paper is organized as follows. In Section \ref{def}, we record all of the preliminary definitions that are necessary for the rest of the paper. In Section \ref{clas}, we consider classical forest-Wilf equivalences, proving Conjecture \ref{gp71} and Theorem \ref{213twist}. In Section \ref{cons}, we examine c-forest-Wilf equivalences, first defining the \emph{grounded permutations} that form the basis of our proofs to Theorems \ref{cclass} and \ref{cequiv}. We then find constraints on strong c-Wilf equivalences between grounded permutations along the lines of \cite[Theorem 1.5]{GP} and prove Theorem \ref{fssequ}, both by enumerating clusters.

\section{Definitions and Notations}
\label{def}
We begin by defining all of the notions of pattern avoidance and rooted forests that we will use throughout this paper.

\begin{defn}
\label{rfdef}
A \emph{rooted labeled forest} on a set $S$ of integers is a forest on $|S|$ vertices labeled with the elements of $S$ in which every connected component has a distinguished root vertex. Each component then has the structure of an unordered rooted tree, and each vertex has a unique label in $S$. Here, by unordered we mean that the tree only has the structure of a graph with a distinguished vertex, and the order in which the neighbors of any vertex are drawn is irrelevant. In a rooted forest $F$ on $S$, we let $L_F(v)$ denote the label of vertex $v$ and suppress the subscript if it is clear from context.
\end{defn}

The same definition can be made for rooted labeled trees. For the sake of brevity, we will oftentimes refer to rooted labeled trees and forests as trees and forests, respectively, and we will always specify when we refer to other types of trees or forests.

We will make use of some standard terminology for rooted trees and forests. In a rooted tree, the \emph{root} is the distinguished vertex. For each non-root vertex $v$, the \emph{parent} of $v$ is the vertex directly before $v$ in the path from the root to $v$, and every non-root vertex is a \emph{child} of its parent. The \emph{ancestors} of a vertex $v$ are the vertices on the path from the root to $v$, and every vertex is a \emph{descendant} of its ancestors. A \emph{strict} descendant of $v$ is a descendant of $v$ that is not equal to $v$, and similarly a strict ancestor of $v$ is an ancestor of $v$ that is not equal to $v$. Each vertex $v$ in the tree has a \emph{depth}, defined as the number of vertices on the path from the root to $v$. For example, the root has depth $1$. The \emph{depth} of a rooted tree $T$ is the maximal depth of a vertex in $T$. A \emph{leaf} is a vertex with no children. A \emph{subtree} of a rooted tree $T$ is a connected induced subgraph of the $T$, which can be viewed as a rooted tree whose root is the vertex of minimal depth in $T$. The \emph{subtree rooted at $v$} is the subtree consisting of all of the descendants of $v$. We may sometimes add modifiers to this term to talk about different subtrees with root $v$. For example, the subtree of $T$ rooted at $v$ consisting of the non-leaf descendants of $v$ refers to the induced subgraph of $T$ on the descendants of $v$ that are not leaves of $T$, which can be viewed as a rooted tree with root $v$.

All of these terms naturally carry over to rooted forests. When we draw rooted forests, we will connect the roots of each connected component to an extra unlabeled vertex and refer to this vertex as the \emph{root} of the forest. The root of the forest can be thought of as the parent of the roots of its connected components, though it is only drawn for visualization purposes and is not actually in the forest or counted when computing the depth of a vertex. In our drawings of rooted forests and trees, the root will be drawn at the top and each vertex will be drawn above its children. In this way, the path from a vertex to any of its descendants is a downward path, and we can refer to vertices as being \emph{higher} or \emph{lower} than other vertices along such a path. For example, the lowest vertex on a given path in the rooted forest is the vertex in the path of maximal depth.

\begin{figure}[ht]
\centering
\scalebox{0.7}
{\forestset{filled circle/.style={
      circle,
      text width=4pt,
      fill,
    },}
\begin{forest}
for tree={filled circle, inner sep = 0pt, outer sep = 0 pt, s sep = 1 cm}
[, 
    [, edge label={node[left]{5}}
        [, edge label={node[left]{3}}
            [, edge label={node[left]{10}}
                [, edge label={node[left]{1}}
                ]
            ]
            [, edge label={node[right]{2}}
            ]
        ]
        [, edge label={node[right]{4}}
        ]
        [, edge label={node[right]{7}}
        ]
    ]
    [, edge label={node[right]{8}}
    ]
    [, edge label={node[right]{12}}
        [, edge label={node[left]{6}}
            [, edge label={node[left]{11}}
            ]
        ]
        [, edge label={node[left]{9}}
        ]
    ]
]
\end{forest}}
\caption{A rooted labeled forest on $[12]$. The root vertex is drawn but not in the forest. We generally draw forests so that the labels of the children of every vertex are sorted in increasing order.}
\label{rlbfig}
\end{figure}
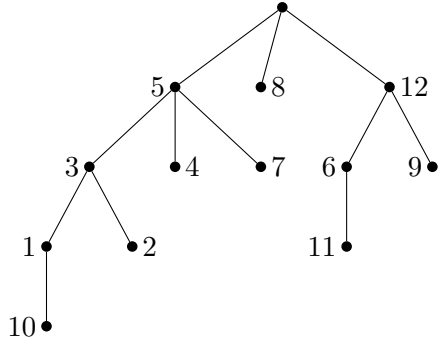

We view rooted trees as trees with a distinguished vertex and rooted forests as a set of rooted trees. Thus, forests may be empty (have $0$ vertices), but trees will always be nonempty. An \emph{increasing} forest is a rooted labeled forest in which every vertex has a smaller label than its children so that the sequence of labels along any downward path in the forest is increasing. The same definition extends to trees, and we can define decreasing forests and trees analogously.

\begin{defn}
\label{instdef}
An \emph{instance} of a pattern $\pi=\pi(1)\cdots\pi(k)$ in a rooted forest $F$ is a sequence of vertices $v_1,\ldots,v_k$ such that $v_i$ is an ancestor of $v_{i+1}$ for all $1\le i<k$ and $L(v_1),\ldots,L(v_k)$ is in the same relative order as $\pi$. We define a \emph{consecutive instance} in the same way, except we require that $v_i$ is a parent of $v_{i+1}$ instead of an ancestor so that $v_1,\ldots,v_k$ forms a downward path in the forest.
\end{defn}

\begin{defn}
\label{avodef}
A forest $F$ \emph{\emph{(}consecutively\emph{)} contains} a pattern $\pi$ if there exists a (consecutive) instance of $\pi$ in $F$, and it \emph{\emph{(}consecutively\emph{)} avoids} a set $S$ of patterns if it does not contain any (consecutive) instance of $\pi$ for all $\pi\in S$.
\end{defn}

We will oftentimes drop the braces when referring to containing or avoiding a specific singleton set $S$. The word classically may be used to describe non-consecutive avoidance or containment, so to classically avoid a set is to avoid a set as in Definition \ref{avodef}. For example, the forest in Figure \ref{rlbfig} contains $321$ through the instance $5,3,1$, and it consecutively contains $312$ through the consecutive instance $12,6,11$. It avoids $123$ and consecutively avoids $321$, but it does not classically avoid $321$. A forest $F$ avoiding a set $S$ of patterns can be viewed as having the property that for every path from the root of $F$ to a leaf of $F$, the sequence of labels avoids $S$ in the sense of pattern avoidance for permutations and sequences. In Sections \ref{clas} and \ref{cons}, we exclusively work with one of non-consecutive and consecutive avoidance, so in those sections and in general when the context is clear, we omit the word consecutive.

\begin{defn}
\label{fwdef}
Two sets $S$ and $S'$ of patterns are \emph{\emph{(}c-\emph{)}forest-Wilf equivalent} if for all $n\ge0$, the number of forests on $[n]$ (consecutively) avoiding $S$ is equal to the number of forests on $[n]$ (consecutively) avoiding $S'$. Two patterns $\pi$ and $\pi'$ are \emph{strongly c-forest-Wilf equivalent} if for all $m,n\ge0$, the number of forests on $[n]$ with exactly $m$ consecutive instances of $\pi$ is equal to the number of forests on $[n]$ with exactly $m$ consecutive instances of $\pi'$.
\end{defn}

As before, we will never work with forest-Wilf equivalence and c-forest-Wilf equivalence in the same section, so for the sake of brevity we will often refer to these concepts as equivalence when the context is clear. However, we will never omit the word strong(ly) or super-strong(ly). (We will define super-strong c-forest-Wilf equivalence in Section \ref{cons}.) We will respectively denote by $S\sim S'$, $S\stackrel c\sim S'$, $S\stackrel{sc}\sim S'$, and $S\stackrel{ssc}\sim S'$ the forest-Wilf equivalence, c-forest-Wilf equivalence, strong c-forest-Wilf equivalence, and super-strong c-forest-Wilf equivalence of two sets $S$ and $S'$ of patterns. We also extend this notation to $\pi\sim \pi'$, $\pi\stackrel c\sim \pi'$, $\pi\stackrel{sc}\sim \pi'$, and $\pi\stackrel{ssc}\sim \pi'$ for equivalences between two individual patterns $\pi$ and $\pi'$.

Recall from Section \ref{intro} that the complement $\overline\pi$ and twist $\widetilde\pi$ of a pattern $\pi=\pi(1)\cdots\pi(k)$ are $k+1-\pi(1),\ldots,k+1-\pi(k)$ and $\pi(1),\ldots,\pi(k-2),\pi(k),\pi(k-1)$, respectively. In other words, the complement is obtained by inverting the order of the elements and the twist is obtained by switching the last two elements. As noted in \cite[Proposition 1]{AA}, given a rooted forest $F$ on $[n]$, we may consider the rooted forest $\overline F$ defined as follows: the underlying unlabeled forest structure will be the same, but $L_{\overline F}(v)=n+1-L_F(v)$ for all vertices $v$. In other words, for all $a\in[n]$, we switch the labels $a$ and $n+1-a$. Note that any instance of a pattern $\pi$ in $F$ will become an instance of $\overline\pi$ in $\overline F$ under this relabelling, so we automatically have that $\{\pi_1,\ldots,\pi_m\}$ and $\{\overline\pi_1,\ldots,\overline\pi_m\}$ are equivalent in all senses described before, by complementation of forest labels and patterns. Such an equivalence between $\pi$ and $\pi'$ where $\pi=\pi'$ or $\pi'=\overline\pi$ is \emph{trivial}.

\section{Classical Wilf Equivalences in Forests}
\label{clas}
In this section we focus on classical forest-Wilf equivalences, the goal being to prove Conjecture \ref{gp71} along with its generalizations. We will consider parts (i) and (ii) in Subsection \ref{gp71i} and part (iii) in Subsection \ref{gp71iii}. The general strategy in both cases is to make use of the \emph{forest Simion-Schmidt bijection}, a forest analogue of the classical bijection between permutations avoiding 123 and 132 given by Simion and Schmidt in 1985 in \cite{SS}. An analogue of this bijection for posets arose as a special case of the work \cite{HW} of Hopkins and Weiler on pattern avoidance in posets. It was first formulated specifically for rooted labeled forests by Anders and Archer in \cite{AA}. Garg and Peng used a variant of this bijection in \cite{GP} to prove Theorem \ref{gptwist}, introducing the \emph{shuffle} and \emph{antishuffle} maps on forests. We first recall a special case of their bijection and refer the reader to \cite[Section 4.1]{GP} for the proofs of the following statements.

\begin{defn}
\label{tdmdef}
In a rooted forest on $[n]$, a \emph{top-down minimum} is a vertex $v$ such that $L(u)\ge L(v)$ for all ancestors $u$ of $v$.
\end{defn}

Note that a forest avoids $123$ if and only if for all vertices $v$ of the forest that are not top-down minima, $L(v)\ge L(w)$ for all descendants $w$ of $v$. Similarly, a forest avoids $132$ if and only if for all vertices $v$ of the forest that are not top-down minima, if $u$ is the lowest ancestor of $v$ that is a top-down minimum, then $L(v)\le L(w)$ or $L(u)\ge L(w)$ for all descendants $w$ of $v$. Figure \ref{123132fig} shows an example of a $123$-avoiding forest and a $132$-avoiding forest. 

\begin{defn}
\label{shufdef}
Suppose that $v$ is a vertex in a labeled forest on $[n]$ that is not a top-down minimum. The \emph{shuffle} and \emph{antishuffle} operations on $v$ are permutations of the labels among the descendants of $v$ such that the relative order of the labels among the strict descendants of $v$ is preserved. This is accomplished by replacing the label of $v$ with the label of one of its descendants and relabelling the strict descendants of $v$ to preserve the relative order. For a shuffle, we relabel $v$ with the maximum value of $L(w)$ among descendants $w$ of $v$. For an antishuffle, we relabel $v$ with the minimum value of $L(w)$ among descendants $w$ of $v$ that is greater than $L(u)$, where $u$ is the lowest ancestor of $v$ that is a top-down minimum.
\end{defn}

Note that shuffles fix the labels and structure of a $123$-avoiding forest and antishuffles fix the labels and structure of a $132$-avoiding forest. The operation $\alpha$ on a $132$-avoiding forest $F$ on $[n]$ is defined as follows: shuffle the vertices of $F$ that are not top-down minima one by one, ordering the vertices by depth from highest to lowest and breaking ties arbitrarily. In this way, we shuffle some leaf of the tree first and the root of the tree last, and we never change which vertices in the tree are top-down minima. The operation $\beta$ on a $123$-avoiding forest $F$ on $[n]$ is defined as follows: antishuffle the vertices of $F$ that are not top-down minima one by one, ordering the vertices by depth from lowest to highest and breaking ties arbitrarily. In this way, we antishuffle the root of the tree first and some leaf of the tree last. The order in which we shuffle or antishuffle vertices of some fixed depth does not matter, as their subtrees do not interact with each other when applying the shuffles or antishuffles. Thus, $\alpha$ and $\beta$ are well-defined. We note that $\alpha$ maps $132$-avoiding forests to $123$-avoiding forests and $\beta$ maps $123$-avoiding forests to $132$-avoiding forests. Furthermore, $\alpha$ and $\beta$ are inverses. In fact, the corresponding steps of $\alpha$ and $\beta$ are all inverses. By this, we mean that if in a step of $\alpha$ we shuffle a vertex $v$ so that the subtree $T$ rooted at $v$ yields a tree $T'$ that only differs from $T$ by labels, the corresponding step of $\beta$ that antishuffles $v$ is applied to $T'$ and results in $T$, and vice versa. The maps $\alpha$ and $\beta$ are also clearly structure-preserving, in that they only permute the labels of the underlying rooted forest. Figure \ref{123132fig} shows an example of an application of $\alpha$ and $\beta$ with three shuffles or antishuffles. This case is particularly simple because the only nontrivial shuffles and antishuffles are at the red, blue, and green vertices, and shuffles and antishuffles applied to other vertices do not change the forest.

\begin{figure}[ht]
\centering
\scalebox{0.7}
{\forestset{filled circle/.style={
      circle,
      text width=4pt,
      fill,
    },}
\begin{minipage}[b]{0.4\linewidth}
\centering
\begin{forest}
for tree={filled circle, inner sep = 0pt, outer sep = 0 pt, s sep = 1 cm}
[,
    [, draw, fill=white, edge label={node[left]{7}}
        [, draw, fill=white, edge label={node[right]{2}}
            [, fill=blue, edge label={node[right]{8}}, name=left8
                [, edge label={node[below]{3}}, name=left3
                ]
                [, edge label={node[below]{4}}, name=left4
                ]
            ]
        ]
        [, fill=red, edge label={node[left]{12}}, name=left12
            [, draw, fill=white, edge label={node[left]{5}}
                [, edge label={node[left]{11}}, name=left11
                ]
            ]
            [, edge label={node[left]{10}}, name=left10
                [, draw, fill=white, edge label={node[left]{1}}
                    [, edge label={node[below]{6}}, name=left6
                    ]
                    [, edge label={node[right]{9}}, name=left9
                    ]
                ]
            ]
        ]
    ]
    [, draw, fill=white, edge label={node[right]{13}}
    ]
]
\draw[->,blue] (left3) to[out=90,in=180] (left8);
\draw[->,blue] (left8) to[out=-90,in=135] (left4);
\draw[->,blue] (left4) to (left3);
\draw[->,red] (left12) to (left11);
\draw[->,red] (left11) to (left10);
\draw[->,red] (left10) to (left9);
\draw[->,red] (left9) to[out=90,in=0] (left12);
\end{forest}
\end{minipage}
\begin{minipage}[b]{0.4\linewidth}
\centering
\begin{forest}
for tree={filled circle, inner sep = 0pt, outer sep = 0 pt, s sep = 1 cm, outer xsep = 0 pt}
[,
    [, draw, fill=white, edge label={node[left]{7}}
        [, draw, fill=white, edge label={node[right]{2}}
            [, fill=blue, edge label={node[right]{3}}, name=right3
                [, edge label={node[below]{4}}, name=right4
                ]
                [, edge label={node[below]{8}}, name=right8
                ]
            ]
        ]
        [, fill=red, edge label={node[left]{9}}, name=right9
            [, draw, fill=white, edge label={node[left]{5}}
                [, edge label={node[left]{12}}, name=right12
                ]
            ]
            [, fill=green, edge label={node[left]{11}}, name=right11
                [, draw, fill=white, edge label={node[left]{1}}
                    [, edge label={node[below]{6}}, name=right6
                    ]
                    [, edge label={node[right]{10}}, name=right10
                    ]
                ]
            ]
        ]
    ]
    [, draw, fill=white, edge label={node[right]{13}}
    ]
]
\draw[->,blue] (right3) to[out=180,in=90] (right4);
\draw[->,blue] (right4) to (right8);
\draw[->,blue] (right8) to[out=135,in=-90] (right3);
\draw[->,red] (right9) to[out=0,in=90] (right10);
\draw[->,red] (right10) to (right11);
\draw[->,red] (right11) to (right12);
\draw[->,red] (right12) to (right9);
\draw[<->,green] (right10) to[out=45,in=0] (right11);
\end{forest}
\end{minipage}
\begin{minipage}[b]{0.4\linewidth}
\centering
\begin{forest}
for tree={filled circle, inner sep = 0pt, outer sep = 0 pt, s sep = 1 cm, outer xsep = 0 pt}
[,
    [, draw, fill=white, edge label={node[left]{7}}
        [, draw, fill=white, edge label={node[right]{2}}
            [, edge label={node[right]{3}}, name=right3
                [, edge label={node[below]{4}}, name=right4
                ]
                [, edge label={node[below]{8}}, name=right8
                ]
            ]
        ]
        [, edge label={node[left]{9}}, name=right9
            [, draw, fill=white, edge label={node[left]{5}}
                [, edge label={node[left]{12}}, name=right12
                ]
            ]
            [, fill=green, edge label={node[left]{10}}, name=right11
                [, draw, fill=white, edge label={node[left]{1}}
                    [, edge label={node[below]{6}}, name=right6
                    ]
                    [, edge label={node[right]{11}}, name=right10
                    ]
                ]
            ]
        ]
    ]
    [, draw, fill=white, edge label={node[right]{13}}
    ]
]
\draw[<->,green] (right10) to[out=45,in=0] (right11);
\end{forest}
\end{minipage}}
\caption{The forest on the left avoids $123$, and the forest on the right avoids $132$. The top-down minima have been drawn with unfilled vertices and are the same for both forests. The red, blue, and green cycles show how vertex labels are permuted when the red, blue, and green vertices are shuffled or antishuffled (in this order from left to right for $\beta$ and in the opposite order from right to left for $\alpha$). These forests correspond to each other under $\alpha$ and $\beta$. Between the left and center forests, the red and blue vertices are (anti)shuffled, and between the center and right forests, the green vertex is (anti)shuffled.}
\label{123132fig}
\end{figure}

Finally, we record the following result, which is a special case of \cite[Theorem 4.8]{GP}.

\begin{thm}
\label{gpshuf213}
The maps $\alpha$ and $\beta$ restrict to maps of a bijection between rooted forests on $[n]$ avoiding $\{213,123\}$ and rooted forests on $[n]$ avoiding $\{213,132\}$.
\end{thm}

\subsection{Restricting the forest Simion-Schmidt bijection}
\label{gp71i}
\hspace*{\fill} \\
We will now prove parts (i) and (ii) of Conjecture \ref{gp71}.

Recall that we wish to show that $\{123,2413\}\sim\{132,2314\}$ and $\{123,3142\}\sim\{132,3124\}$. Our approach is to show that the maps $\alpha$ and $\beta$ restrict to inverse maps of a bijection between rooted forests on $[n]$ avoiding $\{123,X\}$ and rooted forests on $[n]$ avoiding $\{132,Y\}$, where we have that $(X,Y)\in\{(2413,2314),(3142,3124)\}$. Once we do so, the conjecture follows. To show this, for each $(X,Y)$ we identify a property $P$ of rooted forests such that a $123$-avoiding forest avoids $X$ if and only if it satisfies $P$, a $132$-avoiding forest avoids $Y$ if and only if it satisfies $P$, and $P$ is preserved by shuffles and antishuffles. This is clearly sufficient for demonstrating the following theorem.

\begin{thm}
\label{thm71i}
We have that the classical forest-Wilf equivalences $\{123,2413\}\sim\{132,2314\}$ and $\{123,3142\}\sim\{132,3124\}$ hold.
\end{thm}

For the rest of this subsection, we will refer to vertices that are top-down minima as TDM vertices and other vertices as non-TDM vertices.

\begin{defn}
\label{specdef}
A non-TDM vertex $v$ is \emph{special} if the path from the root of the forest to $v$ contains (not necessarily consecutive) vertices $v_1$, $v_2$, $v_3$, and $v_4$, in that order, such that $v_1$ and $v_3$ are TDM and $v_2$ and $v_4$ are non-TDM (for example, we may take $v_4=v$).
\end{defn}

\begin{defn}
\label{ceildef}
Let the \emph{ceiling} of a special vertex $v$ be its lowest ancestor $u$ such that the path from $u$ to $v$ contains (not necessarily consecutive) vertices $v_1$, $v_2$, $v_3$, and $v_4$, in that order, such that $v_1$ and $v_3$ are TDM and $v_2$ and $v_4$ are non-TDM (so we necessarily have $v_1=u$ and $v_4=v$).
\end{defn}

Any path starting from the root of the forest will consist of vertices \[v_{0,1},\ldots,v_{0,m_0},w_{0,1},\ldots,w_{0,n_0},v_{1,1},\ldots,v_{1,m_1},w_{1,1},\ldots,w_{1,n_1},\ldots\] where $v_{i,j}$ and $w_{i,j}$ respectively denote TDM and non-TDM vertices. The special vertices along this path are $w_{i,j}$ for $i>0$, and the special vertices $w_{i,1},\ldots,w_{i,n_i}$ all have the same ceiling $v_{i-1,m_{i-1}}$.

\begin{defn}
\label{p1def}
A \emph{$P_1$-forest} is a forest in which the following holds: for every special vertex $v$ with ceiling $u$, $L(u)>L(v)$.
\end{defn}

Here, $P_1$ is the key property $P$ for $(X,Y)=(2413,2314)$ mentioned at the beginning of this subsection.

\begin{exmp}
\label{p1ex}
Figure \ref{p1fig} shows a forest with labels ommitted, where TDM vertices have been colored black and non-TDM vertices have been colored white. The special vertices are $v_1$ and $v_2$, with respective ceilings $u_1$ and $u_2$.
\end{exmp}

\begin{figure}[ht]
\centering
\scalebox{0.7}
{\forestset{filled circle/.style={
      circle,
      text width=4pt,
      fill,
    },}
\begin{forest}
for tree={filled circle, inner sep = 0pt, outer sep = 0 pt, s sep = 1 cm}
[,
    [, edge label={node[below]{$u_1$}}
        [, draw, fill=white
            [, draw, fill=white
                [,
                    [,
                    ]
                ]
            ]
            [, draw, fill=white
                [,
                    [, draw, fill=white, edge label={node[below]{$v_1$}}
                    ]
                    [,
                    ]
                ]
            ]
        ]
        [, edge label={node[right]{$u_2$}}
            [,
                [, draw, fill=white
                    [, draw, fill=white
                    ]
                    [, draw, fill=white
                    ]
                ]
            ]
            [, draw, fill=white
                [,
                    [, draw, fill=white, edge label={node[below]{$v_2$}}
                    ]
                ]
            ]
        ]
    ]
]
\end{forest}}
\caption{The special vertices $v_1$ and $v_2$ have respective ceilings $u_1$ and $u_2$.}
\label{p1fig}
\end{figure}

\begin{lem}
\label{123p1}
A $123$-avoiding forest avoids $2413$ if and only if it is a $P_1$-forest.
\end{lem}

\begin{proof}
Recall that a forest avoids $123$ if and only if every non-TDM vertex has the largest label among its descendants.

In one direction, we show that if a forest avoids $123$ and contains $2413$, then it contains a special vertex whose label is larger than its ceiling's. Suppose that $v_1,v_2,v_3,v_4$ is an instance of $2413$ in the forest, so $v_1,v_2,v_3,v_4$ lie in that order on a path from the root and $L(v_3)<L(v_1)<L(v_4)<L(v_2)$. Since $L(v_1)<L(v_2)$ and $L(v_3)<L(v_4)$, $v_2$ and $v_4$ are non-TDM. Additionally, $v_1$ and $v_3$ are TDM since they are not greater than all of their strict descendants. Thus, $v_4$ is a special vertex, and its ceiling $u$ is a TDM descendant of $v_1$. Then we have $L(u)<L(v_1)<L(v_4)$, so $v_4$ has a larger label than its ceiling, as desired.

In the other direction, we show that if a forest avoids $123$ and contains a special vertex whose label is larger than its ceiling's, then it contains $2413$. Let $v$ be this special vertex with ceiling $u$, and let $a$ and $b$ be vertices along the path from $u$ to $v$ such that $u$ and $b$ are TDM, $a$ and $v$ are non-TDM, and $u,a,b,v$ appear in that order. We may assume $a$ is a child of $u$, as $u$'s child is non-TDM and appears before $b$ and $v$. As $b$ is TDM, $L(b)<L(u),L(a)$, and as $a$ is non-TDM, $L(a)>L(b),L(v)$. Since $a$ is non-TDM and the child of the non-TDM vertex $u$, $L(u)<L(a)$. By assumption, $L(v)>L(u)$, so $L(a)>L(v)>L(u)>L(b)$. Hence, $u,a,b,v$ form an instance of $2413$, as desired.
\end{proof}

\begin{lem}
\label{132p1}
A $132$-avoiding forest avoids $2314$ if and only if it is a $P_1$-forest.
\end{lem}

\begin{proof}
Recall that a forest avoids $132$ if and only if every non-TDM vertex has the smallest label among its descendants that are greater than its lowest TDM ancestor.

In one direction, we show that if a forest avoids $132$ and contains $2314$, then it contains a special vertex whose label is larger than its ceiling's. Suppose that $v_1,v_2,v_3,v_4$ is an instance of 2314 in the forest, so $v_1,v_2,v_3,v_4$ lie in that order on a path from the root and $L(v_3)<L(v_1)<L(v_2)<L(v_4)$. Since $L(v_2)>L(v_1)$ and $L(v_4)>L(v_3)$, $v_2$ and $v_4$ are non-TDM. Let $v$ be the lowest TDM ancestor of $v_2$, so $L(v)<L(v_2)$. As $L(v_3)<L(v_2)$, we must have that $L(v)>L(v_3)$ or else $v_2$ would not have the smallest possible label greater than $L(v)$. Since $L(v_3)<L(v)<L(v_2)$, we may assume that $v_1=v$, so $v_1$ is the lowest TDM ancestor of $v_2$. Since $L(v_1)>L(v_3)$, there exists a TDM vertex on the path from $v_1$ to $v_3$ different from $v_1$. This TDM vertex must come after $v_2$, so we may assume that it is $v_3$. Thus, $v_1$ and $v_3$ are TDM and $v_2$ and $v_4$ are non-TDM, which means that $v_4$ is a special vertex, and its ceiling $u$ is a TDM descendant of $v_1$. Then we have $L(u)<L(v_1)<L(v_4)$, so $v_4$ has a larger label than its ceiling, as desired.

In the other direction, we show that if a forest avoids $132$ and contains a special vertex whose label is larger than its ceiling's, then it contains $2314$. Let $v$ be this special vertex with ceiling $u$, and let $a$ and $b$ be vertices along the path from $u$ to $v$ such that $u$ and $b$ are TDM, $a$ and $v$ are non-TDM, and $u,a,b,v$ appear in that order. We may take $a$ to be $u$'s child, as $u$'s child is non-TDM and appears before $b$ and $v$. As $b$ is TDM, $L(b)<L(u),L(a)$. By assumption, $L(u)<L(v)$, so since $a$ is non-TDM and its parent $u$ is TDM, $L(u)<L(a)<L(v)$. Putting this together, we have $L(v)>L(a)>L(u)>L(b)$, so $u,a,b,v$ form an instance of $2314$, as desired.
\end{proof}

\begin{lem}
\label{shufp1}
Applying a shuffle or antishuffle to a $P_1$-forest results in a $P_1$-forest.
\end{lem}

\begin{proof}
For each TDM vertex $v$, we define the subtree $T_v$ as the maximal subtree of the forest rooted at $v$ whose only TDM vertex is $v$, where maximal is in the sense of containment. For each non-TDM vertex $u$, we define the subtree $T_u$ as the tree consisting of the descendants of $u$ in $T_v$, where $v$ is the lowest TDM ancestor of $u$. We may view $T_u$ as a subtree of $T_v$, and $T_u$ only consists of non-TDM vertices. Note that this definition only depends on which vertices of the forest are TDM and which are non-TDM, and that the forest is then partitioned into the trees $T_v$ for TDM $v$. Figure \ref{p1pffig} gives an example of how $T_u$ and $T_v$ are defined, where TDM vertices have been colored black and non-TDM vertices have been colored white.

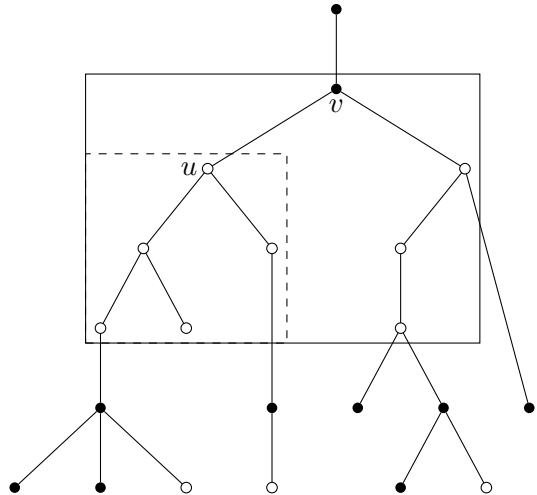
\begin{figure}[ht]
\centering
\scalebox{0.7}
{\forestset{filled circle/.style={
      circle,
      text width=4pt,
      fill,
    },}
\begin{forest}
for tree={filled circle, inner sep = 0pt, outer sep = 0 pt, s sep = 1 cm}
[,
    [, edge label={node[below]{$v$}}, tikz={\node[draw,fit=()(!111)(!l)]{};}
        [, draw, fill=white, edge label = {node[left]{$u$}}, tikz={\node[draw,dashed,fit=()(!11)(!l)]{};}
            [, draw, fill=white
                [, draw, fill=white
                    [, tier=b
                        [,
                        ]
                        [,
                        ]
                        [, draw, fill=white
                        ]
                    ]
                ]
                [, draw, fill=white
                ]
            ]
            [, draw, fill=white
                [, tier=b
                    [, draw, fill=white
                    ]
                ]
            ]
        ]
        [, draw, fill=white
            [, draw, fill=white
                [, draw, fill=white
                    [, tier=b
                    ]
                    [, tier=b
                        [,
                        ]
                        [, draw, fill=white
                        ]
                    ]
                ]
            ]
            [, tier=b]
        ]
    ]
]
\end{forest}}
\caption{The solid box encloses $T_v$, and the dashed box encloses $T_u$.}
\label{p1pffig}
\end{figure}

The condition that the forest is a $P_1$-forest is equivalent to the condition that for all non-TDM vertices $u$ with lowest TDM ancestor $v$, the descendants of $u$ in $T_u$ have greater labels than the label of $v$, which is greater than the labels of the descendants of $u$ not in $T_u$. In other words, $T_u$ consists exactly of the descendants $w$ of $u$ with $L(w)>L(v)$. Indeed, if a vertex $w$ of $T_u$ has label smaller than $v$, then it would be TDM, a contradiction. If a vertex $w$ not in $T_u$ has label greater than $v$, then the path from $u$ to $w$ contains a TDM vertex, say $x$. Then $v,u,x,w$ appear on a path in that order such that $v$ and $x$ are TDM, $u$ and $w$ are non-TDM, and $L(w)>L(v)$. This contradicts the fact that the forest is a $P_1$-forest, as the label of the ceiling of $w$ is at most the label of $v$. Every non-TDM descendant $w$ of $u$ that is not in $T_u$ is a special vertex, and every special vertex arises in this way. We then have $L(w)<L(v)$, so the special vertex $w$ has a smaller label than its ceiling, which has label at most $f(v)$. The condition that the forest is a $P_1$-forest is then satisfied, so the two conditions are equivalent.

Now, note that when we shuffle or antishuffle a non-TDM vertex $u$, we only permute labels within $T_u$. Indeed, $T_u$ contains all of the descendants of $u$ with labels greater than $L(v)$, which are the only vertices that are permuted in any shuffle or antishuffle. Since we only permute within $T_u$, after the shuffle or antishuffle, the equivalent condition that the forest is a $P_1$-forest remains satisfied, as desired.
\end{proof}

\begin{defn}
\label{segdef}
For each non-TDM vertex $v$, let the \emph{segment} of $v$ be the set of TDM ancestors $u$ of $v$ with $L(u)<L(v)$. Let the \emph{top} and \emph{bottom} of the segment of $v$ be the vertices of the segment with the least and greatest labels.
\end{defn}

We may refer to these as the top and bottom of $v$ instead of the segment of $v$. Note that the segment of $v$ necessarily consists of the TDM vertices along the path from the top of $v$ to $v$. 

\begin{defn}
\label{comparableef}
We say that a pair of vertices is \emph{comparable} if one is an ancestor of the other.
\end{defn}

We may view a forest as a Hasse diagram for a poset, which inspires the term comparable.

\begin{defn}
\label{p2def}
A \emph{$P_2$-forest} is a forest in which the following holds: whenever two comparable non-TDM vertices have segments that intersect, the top of their segments coincide.
\end{defn}

Here, $P_2$ is the key property $P$ for $(X,Y)=(3142,3124)$ mentioned at the beginning of this subsection. One way to visualize segments is to plot the labels along a path from the root to a leaf. For example, Figure \ref{p2fig} shows such a plot when if the label sequence is $8,10,6,4,7,3,9,2,5,1$, with the top-down minima, now left-right minima, colored black. The segment of a non-TDM vertex consists of the TDM vertices that have a smaller $x$- and $y$-coordinate in the plot. The forest that the label sequence in Figure \ref{p2fig} comes from is not a $P_2$-forest.

\begin{figure}[ht]
    \centering
    \scalebox{0.7}
    {\begin{tikzpicture}
        \draw[fill] (1/2.5,8/2.5) circle(2pt);
        \node at (1/2.5-0.2,8/2.5-0.2) {$v_1$};
        \draw[gray,dashed] (2/2.5,0)--(2/2.5,10/2.5)--(0,10/2.5);
        \draw (2/2.5,10/2.5) circle(2pt);
        \node at (2/2.5-0.2,10/2.5-0.2) {$v_2$};
        \draw[fill] (3/2.5,6/2.5) circle(2pt);
        \node at (3/2.5-0.2,6/2.5-0.2) {$v_3$};
        \draw[fill] (4/2.5,4/2.5) circle(2pt);
        \node at (4/2.5-0.2,4/2.5-0.2) {$v_4$};
        \draw[gray,dashed] (5/2.5,0)--(5/2.5,7/2.5)--(0,7/2.5);
        \draw (5/2.5,7/2.5) circle(2pt);
        \node at (5/2.5-0.2,7/2.5-0.2) {$v_5$};
        \draw[fill] (6/2.5,3/2.5) circle(2pt);
        \node at (6/2.5-0.2,3/2.5-0.2) {$v_6$};
        \draw[gray,dashed] (7/2.5,0)--(7/2.5,9/2.5)--(0,9/2.5);
        \draw (7/2.5,9/2.5) circle(2pt);
        \node at (7/2.5-0.2,9/2.5-0.2) {$v_7$};
        \draw[fill] (8/2.5,2/2.5) circle(2pt);
        \node at (8/2.5-0.2,2/2.5-0.2) {$v_8$};
        \draw[gray,dashed] (9/2.5,0)--(9/2.5,5/2.5)--(0,5/2.5);
        \draw (9/2.5,5/2.5) circle(2pt);
        \node at (9/2.5-0.2,5/2.5-0.2) {$v_9$};
        \draw[fill] (10/2.5,1/2.5) circle(2pt);
        \node at (10/2.5-0.2,1/2.5-0.2) {$v_{10}$};
    \end{tikzpicture}}
    \caption{The segments of $v_5$ and $v_9$ intersect, but their tops $v_3$ and $v_4$ are different.}
    \label{p2fig}
\end{figure}
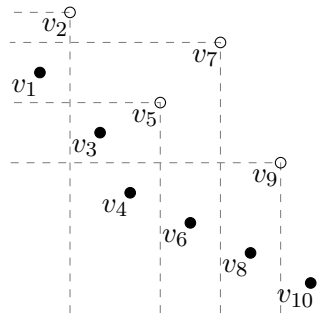

\begin{lem}
\label{123p2}
A $123$-avoiding forest avoids $3142$ if and only if it is a $P_2$-forest.
\end{lem}

\begin{proof}
As before, a forest avoids $123$ if and only if every non-TDM vertex has the largest label among its descendants.

In one direction, we show that if a forest avoids $123$ and contains $3142$, then it contains a pair of comparable non-TDM vertices whose segments intersect but do not have the same top. Suppose that $v_1,v_2,v_3,v_4$ is an instance of 3142 in the forest, so $v_1,v_2,v_3,v_4$ lie in that order on a path from the root and $L(v_2)<L(v_4)<L(v_1)<L(v_3)$. Since $L(v_3),L(v_4)>L(v_2)$, $v_3$ and $v_4$ are non-TDM. Since $L(v_1),L(v_2)<L(v_3)$, $v_1$ and $v_2$ are not greater than all of their strict descendants and are thus TDM. Now the segment of $v_3$ contains $v_1$ and $v_2$, but the segment of $v_4$ contains $v_2$ but not $v_1$, so $v_3$ and $v_4$ form the desired pair of comparable non-TDM vertices.

In the other direction, we show that if a forest avoids $123$ and contains a pair of comparable non-TDM vertices whose segments intersect but do not have the same top, then it contains $3142$. Suppose that $u$ and $v$ form such a comparable pair with $u$ an ancestor of $v$. Then, since $u$ is non-TDM, we have that $L(u)>L(v)$, which means that $u$'s top is an ancestor of $v$'s top. As $u$ and $v$ have intersecting segments, $v$'s top must be contained in $u$'s segment. Let $a$ and $b$ be the tops of $u$ and $v$, respectively, so since $a$ and $b$ are TDM and $a$ is $b$'s ancestor, $L(a)>L(b)$. By assumption, $a$ does not lie in $v$'s segment and $b$ lies in $v$'s segment, so $L(a)>L(v)>L(b)$. Also, $a$ and $b$ are in $u$'s segment, so $L(u)>L(a),L(b)$. Consequently, $a,b,u,v$ appear in that order along a path from the root in the forest with $L(u)>L(a)>L(v)>L(b)$, so they form an instance of $3142$, as desired.
\end{proof}

\begin{lem}
\label{132p2}
A $132$-avoiding forest avoids $3124$ if and only if it is a $P_2$-forest.
\end{lem}

\begin{proof}
As before, a forest avoids $132$ if and only if every non-TDM vertex has the smallest label among its descendants that are greater than the label of its lowest TDM ancestor, i.e. its bottom.

In one direction, we show that if a forest avoids $132$ and contains $3124$, then it contains a pair of comparable non-TDM vertices whose segments intersect but do not have the same top. Suppose that $v_1,v_2,v_3,v_4$ is an instance of 3124 in the forest, so $v_1,v_2,v_3,v_4$ lie in that order on a path from the root and $L(v_2)<L(v_3)<L(v_1)<L(v_4)$. Since $L(v_3)>L(v_2)$ and $L(v_4)>L(v_1)$, $v_3$ and $v_4$ are non-TDM. Because $L(v_1)>L(v_3)$ and $v_3$ is non-TDM, there exists a TDM vertex $v$ on the path from $v_1$ to $v_3$. Indeed, the bottom of $v_3$ necessarily lies strictly between $v_1$ and $v_3$. Suppose that there are no TDM vertices on the path between $v$ and $v_3$ other than $v$. We then have that $L(v)<L(v_3)$ so we may assume that $v_2=v$ is the lowest TDM ancestor of $v_3$. Furthermore, if $v_1$ is non-TDM then since $L(v_3)<L(v_1)$, $L(v_3)$ must be less than the label of $v_1$'s bottom, which means that we can replace $v_1$ with its bottom while preserving the 3124 order of $v_1,v_2,v_3,v_4$. Thus, we may assume that $v_1$ and $v_2$ are TDM and $v_3$ and $v_4$ are non-TDM, which means that $v_3$'s segment contains $v_2$ but not $v_1$ while $v_4$'s segment contains $v_1$ and $v_2$. Hence $v_3$ and $v_4$ form the desired pair of comparable non-TDM vertices.

In the other direction, we show that if a forest avoids $132$ and contains a pair of comparable non-TDM vertices whose segments intersect but do not have the same top, then it contains $3124$. Suppose that $u$ and $v$ forms such a comparable pair with $u$ an ancestor of $v$. We must have that $L(u)<L(v)$, as if $L(u)>L(v)$ then $L(v)$ must be smaller than the label of $u$'s bottom, which means that $v$'s top has a label smaller than $u$'s bottom so $u$ and $v$ have disjoint segments, a contradiction. As $u$ and $v$ have intersecting segments, $u$'s top must be contained in $v$'s segment. Let $a$ and $b$ be the tops of $u$ and $v$, respectively. By assumption, $b$ does not lie in $u$'s segment and $a$ lies in $u$'s segment, so $L(a)<L(u)<L(b)$. Also, $a$ and $b$ are in $v$'s segment, so $L(v)>L(a),L(b)$. Consequently, $b,a,u,v$ appear in that order along a path from the root in the forest and satisfy the inequalities $L(v)>L(b)>L(u)>L(a)$, so they form an instance of 3124, as desired.
\end{proof}

\begin{lem}
\label{shufp2}
Applying a shuffle or antishuffle to a $P_2$-forest results in a $P_2$-forest.
\end{lem}

\begin{proof}
It suffices to show that the segment of each non-TDM vertex does not change when we shuffle or antishuffle any non-TDM vertex, as shuffles and antishuffles do not change which vertices are TDMs. Since the segment of a non-TDM vertex only depends on its top, it suffices to show that the top of each non-TDM vertex does not change when we shuffle or antishuffle any non-TDM vertex.

When we shuffle a non-TDM vertex $v$, we replace its label $L(v)$ with the largest label $L(w)$ among the descendants of $v$. Note that if $v\ne w$, then the bottom of $v$ is in the segment of $w$. Thus, $v$ and $w$ have intersecting segments, so they have the same tops. It follows that when we replace the label of $v$ with $f(w)$, the top of $v$ does not change. The same argument works for all of the vertices that change labels during the shuffle. They are the descendants of $v$ whose labels lie between $L(v)$ and $L(w)$. All of these vertices have the same top as $v$ and $w$, and this does not change after the shuffle relabeling. The segments of the other non-TDM descendants of $v$ do not change either.

When we antishuffle a non-TDM vertex $v$, we replace its label $L(v)$ with the smallest label $L(w)$ among the descendants of $v$ that is greater than the label of the bottom of $v$. Note that if $v\ne w$, then the bottom of $v$ is in the segment of $w$. Thus, $v$ and $w$ have intersecting segments, so they have the same tops. It follows that when we replace the label of $v$ with $L(w)$, the top of $v$ does not change. The same argument works for all of the vertices that change labels during the shuffle. They are the descendants of $v$ whose labels lie between $L(v)$ and $L(w)$. All of these vertices have the same top as $v$ and $w$, and this does not change after the antishuffle relabeling. The segments of the other non-TDM descendants of $v$ do not change either.
\end{proof}

With all of these lemmas, we can finish the proof of Theorem \ref{thm71i}. 

\begin{proof}[Proof of Theorem \ref{thm71i}]
As mentioned before, to show $\{123,X\}\sim\{132,Y\}$ for some patterns $X$ and $Y$, it suffices to show that the maps $\alpha$ and $\beta$ of Garg and Peng in \cite{GP} defined earlier restrict to maps between forests avoiding $\{123,X\}$ and forests avoiding $\{132,Y\}$. To do so, it suffices to exhibit a property $P$ of forests such that a $123$-avoiding forest avoids $X$ if and only if it satisfies $P$, a $132$-avoiding forest avoids $Y$ if and only if it satisfies $P$, and $P$ is preserved by shuffles and antishuffles. For $(X,Y)=(2413,2314)$, the property is $P_1$, demonstrated by Lemmas \ref{123p1}, \ref{132p1}, and \ref{shufp1}, and for $(X,Y)=(3142,3124)$, the property is $P_2$, demonstrated by Lemmas \ref{123p2}, \ref{132p2}, and \ref{shufp2}.
\end{proof}

Since $\alpha$ and $\beta$ restrict to bijections between forests avoiding $\{123,2413\}$ and forests avoiding $\{132,2314\}$ as well as between forests avoiding $\{123,3142\}$ and forests avoiding $\{132,3124\}$, they also restrict to bijections between the intersection of these two restrictions.

\begin{cor}
We have the forest-Wilf equivalence $\{123,2413,3142\}\sim\{132,2314,3124\}$.
\end{cor}

\subsection{Forests avoiding $213$ and another pattern}
\label{gp71iii}
\hspace*{\fill} \\
We will now prove Theorem \ref{213twist} by generalizing the bijection given by Theorem \ref{gpshuf213}.

Throughout this subsection, let $\pi=\pi(1)\cdots\pi(k)$ denote a pattern of length $k$ that also satisfies $\pi(k)=\pi(k-1)+1=\pi(k-2)+2$. Suppose that $\pi(j)=1$. Note that Theorem \ref{213twist} in this case is trivial if $\pi$ contains $213$. Thus, suppose that it does not, so $\pi(1),\ldots,\pi(j-1)>\pi(j),\ldots,\pi(k)$.

\begin{defn}
\label{cpdef}
We say that a positive integer $i<j$ is a \emph{checkpoint} if the numbers $\pi(1),\ldots,\pi(i)$ are all greater than the numbers $\pi(i+1),\ldots,\pi(j)$.
\end{defn}

If the points $P_i=(i,\pi(i))$ are plotted, then $i$ is a checkpoint if we can draw two lines parallel to the coordinate axes such that $P_1,\ldots,P_i$ lie in the upper left quadrant and $P_{i+1},\ldots,P_k$ lie in the lower right quadrant formed by the lines.

\begin{figure}[ht]
    \centering
    \scalebox{0.7}
    {\begin{tikzpicture}
        \draw[gray,dashed] (2.5/2.5,0)--(2.5/2.5,11/2.5);
        \draw[gray, dashed] (0,8.5/2.5)--(11/2.5,8.5/2.5);
        \draw[gray,dashed] (5.5/2.5,0)--(5.5/2.5,11/2.5);
        \draw[gray, dashed] (0,5.5/2.5)--(11/2.5,5.5/2.5);
        \draw[fill] (1/2.5,9/2.5) circle(2pt);
        \draw[fill] (2/2.5,10/2.5) circle(2pt);
        \draw[fill] (3/2.5,6/2.5) circle(2pt);
        \draw[fill] (4/2.5,8/2.5) circle(2pt);
        \draw[fill] (5/2.5,7/2.5) circle(2pt);
        \draw[fill] (6/2.5,1/2.5) circle(2pt);
        \draw[fill] (7/2.5,5/2.5) circle(2pt);
        \draw[fill] (8/2.5,2/2.5) circle(2pt);
        \draw[fill] (9/2.5,3/2.5) circle(2pt);
        \draw[fill] (10/2.5,4/2.5) circle(2pt);
    \end{tikzpicture}}
    \caption{The checkpoints of the pattern $9,10,6,8,7,1,5,2,3,4$ are $2$ and $5$.}
    \label{cpfig}
\end{figure}
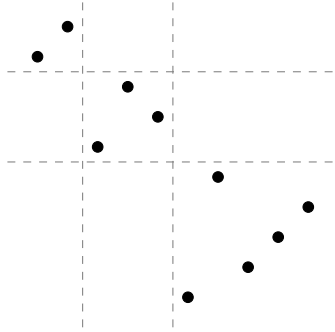

Suppose that the checkpoints of $\pi$ are $a_1<\cdots<a_m$, which is an empty list when $\pi(1)=1$ but always contains $j-1$ when $j>1$. For convenience, we define $a_0=0$.

\begin{defn}
\label{rankdef}
Let the \emph{rank} of a vertex $v$ in a forest with respect to $\pi$ be the greatest positive integer $i\le m$ such that the path from the root of the forest to $v$ contains the pattern $\pi(1)\cdots\pi(a_i)$. If $\pi$ has no checkpoints, or if such an integer $i$ does not exist, then we define the rank to be $0$.
\end{defn}

All trees and forests considered in this subsection will avoid $213$. A forest avoiding $213$ is a set of trees that avoid $213$. In a tree that avoids $213$, the labels greater than the label of the root must come before any labels less than the root on any path from the root to a leaf, a fact shown for example in \cite[Section 3.2]{GP}. Thus, we may decompose the tree into the tree of labels that are at least the label of the root, which we refer to as the \emph{large tree}. The rest of the vertices necessarily form a forest, which we refer to as the \emph{small forest}. Note that the small forest is empty when the root has the smallest label in the tree, but the large tree is always nonempty. We will also consider for a given vertex $v$ in the large tree the vertices in the small forest whose lowest ancestor in the large tree is $v$. These descendants of $v$ form a forest, which we refer to as the \emph{small subforest of $v$}.

\begin{figure}[ht]
    \centering
    \scalebox{0.7}
    {\forestset{filled circle/.style={
          circle,
          text width=4pt,
          fill,
        },}
    \begin{forest}
    for tree={filled circle, inner sep = 0pt, outer sep = 0 pt, s sep = 1 cm}
    [, fill=white
        [, fill=red, edge=white, edge label={node[above, black]{7[1]}}
            [, edge=red, fill=red, edge label={node[left,black]{10[1]}}
                [, fill=blue, edge label={node[left]{3}}, tier=b
                ]
                [, edge=red, fill=red, edge label={node[left,black]{8[2]}}
                ]
                [, edge=red, fill=red, edge label={node[right,black]{9[2]}}, tikz={\node[draw,fit=(!1)(!l1)]{};}
                    [, fill=blue, edge label={node[right]{1}}, tier=b
                        [, edge=blue, fill=blue, edge label={node[right,black]{2}}
                        ]
                    ]
                    [, fill=blue, edge label={node[left]{4}}, tier=b
                        [, edge=blue, fill=blue, edge label={node[left,black]{5}}
                        ]
                    ]
                ]
            ]
            [, fill=blue, edge label={node[right]{6}}, tier=b
            ]
        ]
    ]
    \end{forest}}
    \caption{A $213$-avoiding tree on $[10]$. The large tree on $\{7,8,9,10\}$ is colored red, the small forest on $\{1,2,3,4,5,6\}$ is colored blue, and the small subforest of the vertex labeled $9$ is boxed. Vertices in the large tree are also marked by their ranks with respect to the pattern $54123$ in brackets.}
    \label{rankfig}
\end{figure}
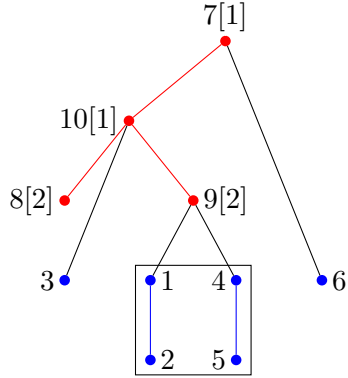

The rank of a vertex $v$ in the large tree of a tree $T$ measures the ``progress'' that the path from the root of $T$ to $v$ makes in containing an instance of $\pi$ by recording the index of the furthest checkpoint the label sequence in this path gets past. Figure \ref{rankfig} gives an example of the decomposition of a $213$-avoiding forest along with ranks of vertices in the large tree with respect to $\pi=54123$. Continuing into the small subforest of $v$, the only thing that matters towards containing $\pi$ is how much progress was already made in the large tree of $T$, i.e. the rank of $v$. Avoiding or containing $\pi$ now becomes a matter of avoiding or containing a truncation of $\pi$ in the small subforest of $v$. This is formalized in the following rank additivity lemma.

\begin{lem}
\label{rankadd}
Suppose that a vertex $v$ of the large tree of a tree $T$ has rank $r$ with respect to $T$ and $\pi$. Let $F$ be the small subforest of $v$ with respect to $T$, and suppose that a vertex $w$ of $F$ has rank $s$ with respect to $F$ and the pattern $\pi(a_r+1)\cdots\pi(k)$. Then $w$ has rank $r+s$ with respect to $T$ and $\pi$. Furthermore, the union of the path from the root of $T$ to $v$ and $F$ avoids $\pi$ if and only if $F$ avoids the pattern $\pi(a_r+1)\cdots\pi(k)$.
\end{lem}

\begin{proof}
We first take care of the case that both $r$ and $s$ are positive. As defined before, we suppose that $\pi$ has checkpoints at $a_1,\ldots,a_m$. Note that $\pi(a_r+1)\cdots\pi(k)$ then has checkpoints at the indices $a_{r+1}-a_r,\ldots,a_m-a_r$. Suppose that $v_1,\ldots,v_{a_r}$ is an instance of $\pi(1)\cdots\pi(a_r)$ on the path from the root of $T$ to $v$ and that $w_1,\ldots,w_{a_{r+s}-a_r}$ is an instance of $\pi(a_r+1)\cdots\pi(a_{r+s})$ on the path from $v$ to $w$, not including $v$. Then as $a_r$ is a checkpoint, we know that $\pi(1),\ldots,\pi(a_r)$ are all greater than $\pi(a_r+1),\ldots,\pi(a_{r+s})$. As $v_1,\ldots,v_{a_r}$ are ancestors of $v$ in the large tree and $w_1,\ldots,w_{a_{r+s}-a_r}$ are in the small subforest of $v$, we have that the labels of $v_1,\ldots,v_{a_r}$ are greater than the labels of $w_1,\ldots,w_{a_{r+s}-a_r}$. Thus, $v_1,\ldots,v_{a_r},w_1,\ldots,w_{a_{r+s}-a_r}$ is an instance of $\pi(1)\cdots\pi(a_{r+s})$ and the rank of $w$ with respect to $T$ and $\pi$ is at least $r+s$. Conversely, suppose that $r+s<m$ (note that we are done if $r+s=m$) and that the rank of $w$ with respect to $T$ and $\pi$ is more than $r+s$. Then, there is an instance $v_1,\ldots,v_x,w_1,\ldots,w_y$ of $\pi(1)\cdots\pi(a_{r+s+1})$ on the path from the root of $T$ to $w$, where the $v_i$ are ancestors of $v$ and the $w_i$ are in the small subforest of $v$. As the labels of $v_1,\ldots,v_x$ are greater than the labels of $w_1,\ldots,w_y$, $x$ must be a checkpoint of $\pi$. As the rank of $v$ with respect to $T$ and $\pi$ is $r$, we have that $x\le a_r$, so $w_1,\ldots,w_y$ contains an instance of $\pi(a_r+1)\cdots\pi(a_{r+s+1})$. But $a_{r+s+1}-a_r$ is the $(s+1)$-st checkpoint of $\pi(a_r+1)\cdots\pi(k)$, contradicting the fact that the rank of $w$ with respect to $F$ and $\pi(a_r+1)\cdots\pi(k)$ is $s$. Thus, the rank of $w$ with respect to $T$ and $\pi$ is $r+s$.

The other cases are trivial or follow similarly. If $r=0$ because $\pi$ has no checkpoints (so $\pi(1)=1$), then all vertices of $T$ have rank $0$. We then have that $s$ is the rank of $w$ with respect to $\pi$ in $F$, which is $0$ as $\pi$ has no checkpoints. If $r=0$ because in the path from the root of $T$ to $v$, there is no instance of $\pi(1)\cdots\pi(a_1)$, then the proof is the same as in the case $r,s>0$. If $s=0$ because $\pi(a_r+1)\cdots\pi(k)$ has no checkpoints, then the rank of $v$ with respect to $\pi$ and $T$ is $r=m$. Then all descendants of $v$ also have rank $m$, and all vertices in $F$ have rank $0$ with respect to $F$ and $\pi(a_r+1)\cdots\pi(k)$ since it has no checkpoints. If $s=0$ because there is no instance of $\pi(a_r+1)\cdots\pi(a_{r+1})$ in the path from $v$ to $w$ excluding $v$, then the proof is the same as the proof for $r,s>0$.

The proof of the second part of the lemma is essentially the same. Any instance of $\pi$ of the form $v_1,\ldots,v_x,w_1,\ldots,w_y$ with $v_i$ in the large tree and $w_i$ in the small subforest of $v$ must satisfy the property that the labels of $v_1,\ldots,v_x$ are greater than the labels of $w_1,\ldots,w_y$. Thus, $x$ is a checkpoint of $\pi$ and is at most $a_r$ since the rank of $v$ with respect to $\pi$ and $T$ is $r$. Then, $w_1,\ldots,w_y$ must contain an instance of $\pi(a_r+1)\cdots\pi(k)$. Conversely, if $w_1,\ldots,w_{k-a_r}$ form an instance of $\pi(a_r+1)\cdots\pi(k)$ in $F$, then concatenating with $v_1,\ldots,v_{a_r}$, an instance of $\pi(1)\cdots\pi(a_r)$ on the path from the root of $T$ to $v$, gives an instance of $\pi$.
\end{proof}

The key idea behind the proof for Theorem \ref{213twist} is a recursive bijection between forests avoiding $\{213,\pi\}$ and forests avoiding $\{213,\widetilde\pi\}$ that preserves the structure of the forest (it only permutes the labels) and rank of each vertex. Lemma \ref{rankadd} allows us to transfer between avoiding $\pi$ and avoiding truncations of $\pi$ in small subforests, which allows us to inductively define such a bijection, with Theorem \ref{gpshuf213} forming our base case. We denote the structure- and rank-preserving map from forests avoiding $\{213,\pi\}$ to forests avoiding $\{213,\widetilde\pi\}$ by $f_\pi$ and the inverse map by $f_{\widetilde\pi}$. Note that it suffices to define such a bijection between trees avoiding $\{213,\pi\}$ and trees avoiding $\{213,\widetilde\pi\}$, as we can then apply this bijection to each tree in the forest. The ranks of the vertices only depend on the relative ordering of the labels within a tree, so the bijection extends properly. We define this bijection recursively on trees $T$ with the following cases, noting that whenever we write $f_\sigma$ or $f_{\widetilde\sigma}$ for some pattern $\sigma$, the last three terms of $\sigma$ are increasing consecutive integers. We may refer to applying $f_\sigma$ to a subforest $F$ of $T$, by which we mean to replace the labels of $F$ by the labels of $f_\sigma(F)$ (note this does not change the label set of $F$). We define $f_\pi$ and $f_{\widetilde\pi}$ by splitting into the following cases, where we are allowed to apply $f_\sigma$ to the forest in cases (3) and (6) by Lemma \ref{rankadd}:

\begin{enumerate}
    \item If $T$ has less than three vertices, we define $f_\pi(T)=f_{\widetilde\pi}(T)=T$ to be the identity map.
    
    \item If $\pi=123$, we define $f_\pi$ and $f_{\widetilde\pi}$ to be $\beta$ and $\alpha$, respectively, where $\alpha$ and $\beta$ are as defined in the beginning of this section.
    
    \item If $\pi(1)=1$ and the label of the root of the tree is $1$, we define $f_\pi(T)$ and $f_{\widetilde\pi}(T)$ to be the result of applying $f_\sigma$ and $f_{\widetilde\sigma}$ to the strict descendants of the root, which is a forest, where $\sigma=\pi(2)-1,\ldots,\pi(k)-1$.
    
    \item If $\pi(1)>1$ and the label of the root of the tree is $1$, we define $f_\pi(T)$ and $f_{\widetilde\pi}(T)$ to be the result of applying $f_\pi$ and $f_{\widetilde\pi}$ to the strict descendants of the root, which is a forest.
    
    \item If $\pi(1)=1$ and the label of the root is greater than $1$, we define $f_\pi(T)$ and $f_{\widetilde\pi}(T)$ to be the result of applying $f_\pi$ and $f_{\widetilde\pi}$ to both the large tree and small forest of $T$, respectively.
    
    \item If $\pi(1)>1$ and the label of the root is greater than $1$, we define $f_\pi(T)$ and $f_{\widetilde\pi}(T)$ as follows. We respectively apply $f_\pi$ and $f_{\widetilde\pi}$ to the large tree of $T$. Then for each vertex $v$ of the large tree of $T$, we do the following. If the rank of $v$ with respect to $T$ and $v$ is $r$, then we respectively apply $f_\sigma$ and $f_{\widetilde\sigma}$ to the small subforest of $v$ in $T$, where $\sigma=\pi(a_r+1)\cdots\pi(k)$.
\end{enumerate}

\begin{prop}
\label{213pi}
The maps $f_\pi$ and $f_{\widetilde\pi}$ are structure- and rank-preserving and are inverses. Furthermore, $f_\pi$ maps $\{213,\pi\}$-avoiding forests to $\{213,\widetilde\pi\}$-avoiding forests and $f_{\widetilde\pi}$ maps $\{213,\widetilde\pi\}$-avoiding forests to $\{213,\pi\}$-avoiding forests.
\end{prop}

\begin{proof}
We proceed using induction on $k$ (the length of $\pi$) and the number of vertices in the tree to show that the defined maps are inverses, noting that the bijection clearly extends to forests with all of the claimed properties. Note that the notion of rank does not depend on $\pi(k-1)$ or $\pi(k)$, so the rank with respect to $\pi$ is the same as the rank with respect to $\widetilde\pi$.

For the base case, if the number of vertices is less than three, then the two maps are both the identity map and are obviously inverses of each other that fix structure and rank. If $\pi=123$, then Theorem \ref{gpshuf213} along with the definitions of $\alpha$ and $\beta$ implies that the defined maps are inverses of each other and preserve structure. All vertices have rank $0$ when $\pi=123$, so the base case follows.

For the inductive step, note first that the maps we define are always inverses of each other, preserve structure, and do not change the large tree of $T$. Indeed, we always either apply $f_\pi$ or $f_{\widetilde\pi}$ to a smaller forest (note the different treatment of the cases when the root has label $1$, corresponding to when the small forest is empty) or apply $f_\sigma$ or $f_{\widetilde\sigma}$, where $\sigma$ is shorter than $\pi$. We always apply $f_\sigma$ to disjoint subforests of $T$. Hence, $f_\pi$ and $f_{\widetilde\pi}$ always invert each other. Indeed, we also see by induction that whether the root of the tree is $1$ or not is also fixed by our maps, so all of the cases of the recursive definitions imply that the maps are inverses of each other. This settles all but the last case defined in the recursion, where the smaller map applied depends on the rank of vertices in $T$. Once we show that the maps preserve rank, this case also follows, and all of the claimed properties of the maps will be proved.

For the inductive step that our defined maps preserve vertex ranks, we split into the cases in our definition of $f_\pi$. Note that there is nothing to show when $\pi(1)=1$ (Cases (3) and (5)), as all vertices have rank $0$ in that case. If $\pi(1)>1$ and the label of the root of the tree is $1$ (Case (4)), then note that the rank of the root is $0$ and the rank of a vertex $v$ with respect to $\pi$ and $T$ is equal to the rank of $v$ with respect to $T$ with the root removed by Lemma \ref{rankadd}. This follows because when the root has label $1$ and $\pi(1)>1$, the root is never part of an instance of $\pi$. If $\pi(1)>1$ and the label of the root of the tree is greater than $1$ (Case (6)), then note that the ranks of all vertices in the large tree are preserved by induction, as the large tree has fewer vertices than $T$. But the ranks with respect to all of the the small subforests and their corresponding truncated patterns are also preserved by induction, so by Lemma \ref{rankadd}, the rank of a vertex of the small forest with respect to $T$ and $\pi$ is preserved as the sum of its rank with respect to the small subforest and the rank of its lowest ancestor in the large tree.

It remains to show that $f_\pi$ maps a tree avoiding $\{213,\pi\}$ to a tree avoiding $\{213,\widetilde\pi\}$. This inductive step argument is essentially the same as the previous one. When the root of $T$ has label $1$ and $\pi(1)>1$ (Case (4)), avoiding $\pi$ is equivalent to avoiding $\pi$ in the forest resulting from removing the root, as the root is never involved in instances of $\pi$. When the root of $T$ has label $1$ and $\pi(1)=1$ (Case (3)), avoiding $\pi$ is equivalent to avoiding $\pi(2)\cdots\pi(k)$ in the forest resulting from removing the root, as every instance of $\pi$ in $T$ gives an instance of $\pi(2)\cdots\pi(k)$ in the forest, while every instance of $\pi(2)\cdots\pi(k)$ in the forest gives an instance of $\pi$ in $T$ by appending the root to the beginning. When the root of $T$ has label greater than 1 and $\pi(1)=1$ (Case (5)), note that instances of $\pi$ must be entirely contained in the large tree or the small forest, since the first vertex of the occurrence must have the smallest label. Finally, when the root of $T$ has label greater than 1 and $\pi(1)>1$ (Case (6)), instances of $\pi$ occur if and only if the small subforests do not contain the corresponding truncations of $\pi$ by Lemma \ref{rankadd}. We have shown avoidance of $\pi$ to be equivalent to a combination of avoidance of $\pi$ and truncations of $\pi$ in disjoint subforests, and the recursion we have defined gives exactly the same equivalent avoidances in the corresponding subforests, so we are done.
\end{proof}

\begin{proof}[Proof of Theorem \ref{213twist}]
Proposition \ref{213pi} shows that the bijection we have defined exists, preserves structure and rank, and maps $\{213,\pi\}$-avoiding forests on $[n]$ to $\{213,\widetilde\pi\}$-avoiding forests on $[n]$, so we have that $\{213,\pi\}\sim\{213,\widetilde\pi\}$.
\end{proof}

\section{Consecutive Wilf Equivalences in Forests}
\label{cons}
In this section, we discuss various notions of c-forest-Wilf equivalence. As mentioned in Section \ref{def}, whenever we refer to instances and equivalences throughout this section, we mean consecutive instances and c-forest-Wilf equivalences, respectively. In Subsection \ref{cfam}, we construct a large family of such c-forest-Wilf equivalences between certain types of patterns we call \emph{grounded permutations} to prove Theorems \ref{cclass} and \ref{cequiv}. Subsection \ref{ground} then discusses some necessary conditions for two grounded permutations to be strongly c-forest-Wilf equivalent, following the methodology of Garg and Peng in their proof of the following theorem from \cite{GP}.

\begin{thm}[{\cite[Theorem 1.5]{GP}}]
\label{gp15}
If $\pi$ and $\pi'$ are patterns of length $k$ such that $\pi\stackrel{sc}\sim\pi'$, then $\pi(1)=\pi'(1)$ or $\pi(1)+\pi'(1)=k+1$.
\end{thm}

Finally, we investigate super-strong c-forest-Wilf equivalence in Subsection \ref{sstriv} and prove Theorem \ref{fssequ}.

We begin by defining some notations and terms specific to the theory of consecutive avoidance. The following notion of a \emph{forest cluster} was introduced by Garg and Peng in \cite{GP}.

\begin{defn}[{\cite[Definition 6.1]{GP}}]
\label{clusdef}
A \emph{$m$-cluster of size $n$} with respect to a pattern $\pi$ is a rooted tree on $[n]$ (a rooted forest on $[n]$ with only one component) with $m$ distinct \emph{highlighted instances} of $\pi$ such that every vertex is a part of a highlighted instance and there is no proper nonempty subset $V$ of vertices such that all highlighted instances either lie completely in $V$ or do not intersect $V$. Informally, the highlighted instances in the cluster are ``connected.''
\end{defn}

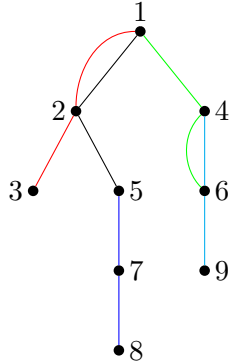
\begin{figure}[ht]
    \centering
    \scalebox{0.7}
        {\forestset{filled circle/.style={
          circle,
          text width=4pt,
          fill,
        },}
    \begin{forest}
    for tree={filled circle, inner sep = 0pt, outer sep = 0 pt, s sep = 1 cm}
    [, fill=white
        [, edge=white, edge label={node[above,black]{1}}, name=v1
            [, edge label={node[left,black]{2}}, name=v2
                [, edge=red, edge label={node[left,black]{3}}
                ]
                [, edge label={node[right,black]{5}}
                    [, edge=blue, edge label={node[right,black]{7}}
                        [, edge=blue, edge label={node[right,black]{8}}
                        ]
                    ]
                ]
            ]
            [, edge=green, edge label={node[right,black]{4}}, name=v4
                [, edge=cyan, edge label={node[right,black]{6}}, name=v6
                    [, edge=cyan, edge label={node[right,black]{9}}
                    ]
                ]
            ]
        ]
    ]
    \path[red,out=180,in=90] (v1.child anchor) edge (v2.parent anchor);
    \path[green,out=-135,in=135] (v4.child anchor) edge (v6.parent anchor);
    \end{forest}}
    \caption{A $5$-cluster of size $9$ with respect to $123$. The instances are the colored paths on $\{1,2,3\}$, $\{1,2,5\}$, $\{1,4,6\}$, $\{4,6,9\}$, and $\{5,7,8\}$.}
    \label{clusfig}
\end{figure}

Figure \ref{clusfig} shows a forest cluster with respect to the pattern $123$, where we use colored paths to denote the different highlighted instances. We do not require that every instance is highlighted. Indeed, note that the instance $2,5,7$ is not highlighted. If it were, then we would have a different forest cluster.

Forest clusters are the analogue of permutation clusters for forests. General cluster methods were introduced by Goulden and Jackson in \cite{GJ} and applied to permutation consecutive pattern avoidance by Elizalde and Noy in \cite{EN}. Forest clusters were introduced by Garg and Peng in \cite{GP}, where their connection to strong equivalence was shown through the following theorem.

\begin{thm}[{\cite[Theorem 6.2]{GP}}]
\label{gpfclust}
Two patterns $\pi$ and $\pi'$ are strongly equivalent if and only if for all $m,n\ge0$, the number of $m$-clusters of size $n$ with respect to $\pi$ is equal to the number of $m$-clusters of size $n$ with respect to $\pi'$.
\end{thm}

Equivalences between patterns are typically proven by showing strong equivalence via the cluster method for both permutations and forests. The proof of Theorem \ref{gpcequiv} in \cite{GP} proceeds by showing that the number of clusters for $1324$ and $1423$ satisfy the same recursion. Our results in this section also follow this method. It would be interesting to see a proof of equivalence without appealing to strong equivalence.

We can refine the notion of equivalence by specifying not only the number of vertices in the cluster but the structure of the tree.

\begin{defn}
\label{clustrucdef}
Two patterns $\pi$ and $\pi'$ satisfy \emph{cluster structure equivalence} if for all (unlabeled) rooted trees $T$ and integers $m\ge0$, the number of clusters on $T$ with $m$ highlighted instances of $\pi$ is equal to the the number of clusters on $T$ with $m$ highlighted instances of $\pi'$.
\end{defn}

We can also extend the notion of \emph{super-strong} equivalence, introduced by Dwyer and Elizalde for permutations in \cite{DE}, to forests.

\begin{defn}
Two patterns $\pi$ and $\pi'$ of length $k$ are \emph{super-strongly} equivalent if for all unlabeled rooted forests $F$ on $n$ vertices with a set $S$ of highlighted paths of length $k$, the number of ways to label the vertices of $F$ with distinct labels in $[n]$ such that $S$ is the set of instances of $\pi$ in $F$ is equal to the number of ways to label the vertices of $F$ with distinct labels in $[n]$ such that $S$ is the set of instances of $\pi'$ in $F$.
\end{defn}

\subsection{A large family of nontrivial equivalences}
\label{cfam}
\hspace*{\fill} \\
In this section, we prove Theorems \ref{cclass} and \ref{cequiv}, which state that there is an equivalence class of size at least $2^{n-4}$ for all $n$ and that a $(1-o(1))^n$-fraction of patterns of length $n$ satisfy a nontrivial equivalence. We do so by first demonstrating a family of nontrivial equivalences through a series of lemmas. Then, we use the properties of this family to obtain Theorems \ref{cclass} and \ref{cequiv} as a result.

First, we define some notions related to a certain class of permutations we will consider. Throughout this subsection, we assume that $\pi=\pi(1)\cdots\pi(k)$ is a pattern of length $k$ that is not the identity.

\begin{defn}
\label{streakdef}
The \emph{streak} of a pattern $\pi=\pi(1)\cdots\pi(k)$ is the largest integer $m$ such that $\pi(i)=i$ for all $1\le i\le m$. 
\end{defn}

Note that if $\pi(1)\ne1$, then the streak of $\pi$ is $0$.

\begin{defn}
\label{heightdef}
For $1<i\le k$, let the \emph{height} of $\pi(i)$ in $\pi$ be the largest integer $m<i$ such that $\pi(i-m+1)<\ldots<\pi(i)$, i.e. the length of the longest consecutive increasing subsequence beginning after $\pi(1)$ and ending at $\pi(i)$. For $\pi(1)\ne j\in[k]$, let the height of $j$ in $\pi$ be the height of $\pi(i)$ in $\pi$, where $\pi(i)=j$. Let the \emph{max height} of $\pi$, be the maximum of the height of $\pi(i)$ over all $1<i\le k$.
\end{defn}

\begin{exmp}
\label{streakex}
The pattern $\pi=123485679$ has streak $4$ and max height $4$, since the heights of $8$ and $9$ in $\pi$ are $4$ due to the maximal increasing consecutive subsequences $2348$ and $5679$. Note that $12348$ is an increasing consecutive subsequence of length $5$, but we do not count it because it starts at the first element of $\pi$.
\end{exmp}

We will be working with a more general class of clusters. Note that the definition of clusters extends to clusters with respect to a set of more than one pattern in the obvious way: we have a rooted tree with some highlighted instances of patterns in the set such that it is not possible to partition the vertices of the tree into two parts so that each highlighted instance has vertices in at most one part.

\begin{defn}
\label{pseudodef}
A \emph{$p$-pseudo $m$-cluster of size $n$} with respect to $\pi$ is an $m$-cluster of size $n$ with respect to the patterns $\{\pi(1)\cdots\pi(k),\pi(2)\cdots\pi(k),\ldots,\pi(p)\cdots\pi(k)\}$ such that each highlighted instance of a pattern that is not $\pi$ contains the root of the cluster when the cluster is viewed as a tree.
\end{defn}

\begin{defn}
\label{primdef}
A \emph{primitive} $m$-cluster of size $n$ with respect to $\pi$ is an $m$-cluster of size $n$ in which all highlighted instances of $\pi$ contain the root of the cluster. A \emph{primitive} $p$-pseudo $m$-cluster of size $n$ with respect to $\pi$ is a $p$-pseudo $m$-cluster of size $n$ in which all highlighted instances of a pattern contains the root of the cluster.
\end{defn}

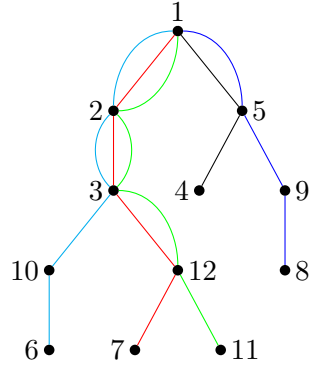
\begin{figure}[ht]
    \centering
    \scalebox{0.7}
    {\forestset{filled circle/.style={
      circle,
      text width=4pt,
      fill,
    },}
    \begin{forest}
    for tree={filled circle, inner sep = 0pt, outer sep = 0 pt, s sep = 1 cm}
    [, fill=white
        [, edge=white, edge label={node[above,black]{1}}, name=v1
            [, edge=red, edge label={node[left,black]{2}}, name=v2
                [, edge=red, edge label={node[left,black]{3}}, name=v3
                    [, edge=cyan, edge label={node[left,black]{10}}, name=v10
                        [, edge=cyan, edge label={node[left,black]{6}}, name=v6
                        ]
                    ]
                    [, edge=red, edge label={node[right,black]{12}}, name=v12
                        [, edge=red, edge label={node[left,black]{7}}, name=v7
                        ]
                        [, edge=green, edge label={node[right,black]{11}}, name=v11
                        ]
                    ]
                ]
            ]
            [, edge label={node[right,black]{5}}, name=v5
                [, edge label={node[left,black]{4}}, name=v4
                ]
                [, edge=blue, edge label={node[right,black]{9}}, name=v9
                    [, edge=blue, edge label={node[right,black]{8}}, name=v8
                    ]
                ]
            ]
        ]
    ]
    \path[blue,out=0,in=90] (v1.child anchor) edge (v5.parent anchor);
    \path[cyan,out=180,in=90] (v1.child anchor) edge (v2.parent anchor);
    \path[cyan,out=-135,in=135] (v2.child anchor) edge (v3.parent anchor);
    \path[green,out=-90,in=0] (v1.child anchor) edge (v2.parent anchor);
    \path[green,out=-45,in=45] (v2.child anchor) edge (v3.parent anchor);
    \path[green,out=0,in=90] (v3.child anchor) edge (v12.parent anchor);
    \end{forest}}
    \caption{A primitive $3$-pseudo $5$-cluster of size $12$ with respect to $12354$. The instances are the colored paths on $\{1,2,3,10,6\}$, $\{1,2,3,12,7\}$, $\{1,2,3,12,11\}$, $\{1,5,4\}$, and $\{1,5,9,8\}$. The blue and black paths, respectively on $\{1,5,9,8\}$ and $\{1,5,4\}$, are instances of $2354$ and $354$.}
    \label{pseudofig}
\end{figure}

When speaking generally, we may omit some of $m,n,p$. We note that (primitive) $1$-pseudo clusters are simply (primitive) clusters, and that for $i<j$, all (primitive) $i$-pseudo clusters are (primitive) $j$-pseudo clusters. In a primitive $m$-cluster, there are exactly $m$ leaves, all at depth $k$.

\begin{defn}
\label{clusheightdef}
In a primitive pseudo cluster $C$, for each vertex $v$ that is not the root of the cluster, the \emph{height} of $v$ in $C$ is the largest integer $m$ such that there exist vertices $v_1,\ldots,v_m=v$, none of which is the root of $C$, such that the labels of $v_1,\ldots,v_m$ are increasing and for all $1\le i<m$, $v_{i+1}$ is $v_i$'s child. In other words, the height of a vertex is the length of the longest consecutive sequence of vertices different from the root of $C$ with increasing labels.
\end{defn}

Note that if $v$ is the $i$th vertex in a highlighted instance of $\pi$ in the cluster, then the height of $v$ in $C$ is the same as the height of $\pi(i)$ in $\pi$. However, the notions of height do not necessarily coincide for vertices not in a highlighted instance of $\pi$.

\begin{defn}
\label{pequivdef}
Two permutations $\pi$ and $\pi'$ of length $k$ are \emph{$p$-equivalent} if for all nonnegative integers $\ell,m,n$ with $\ell\in[n]$ and functions $f:[n]\setminus\{\ell\}\rightarrow[k]$, the number of primitive $p$-pseudo $m$-clusters $C$ of size $n$ whose root is labeled $\ell$ such that for all $i\in[n]\setminus\{\ell\}$, $f(i)$ is equal to the height of the vertex with label $i$ in $C$, is equal for $\pi$ and $\pi'$.
\end{defn}

\begin{defn}
\label{opcdef}
An \emph{ordered primitive cluster} is a primitive cluster whose underlying rooted tree is an ordered rooted tree, in the sense that the children of any vertex are ordered, so if labels between the children of a vertex are permuted a different labeled rooted tree is obtained.
\end{defn}

Thus far, we have been working with unordered forests. We will only consider ordered (i.e. planar) forests in the context of ordered primitive clusters. Note that given the structure of the underlying tree, the ratio of the number of ordered clusters to the number of unordered clusters is the number of automorphisms of the tree, so the notions are closely related.

\begin{defn}
\label{psequivdef}
Two permutations $\pi$ and $\pi'$ of length $k$ are \emph{primitive structure equivalent} if for all nonnegative integers $m$ and $n$ and unlabeled rooted trees $T$ on $n$ vertices with $m$ leaves, the number of labellings of $T$ that are ordered primitive clusters for $\pi$ is equal to the number of labellings of $T$ that are ordered primitive clusters for $\pi'$.
\end{defn}

The family of equivalences we will describe involve a class of permutation we call the \emph{grounded permutations}.

\begin{defn}
\label{groundef}
A permutation $\pi=\pi(1)\cdots\pi(k)\ne1\cdots k$ is \emph{grounded} if there exists a positive integer $m$ such that $\pi(1)\cdots\pi(m)=1\cdots m$ and $\pi(2)\cdots\pi(k)$ avoids the pattern $1\cdots(m+1)$ in the consecutive sense of permutation pattern avoidance.
\end{defn}

In other words, a pattern $\pi$ is grounded if it is not the identity and its max height is at most its streak. Since the max height of $\pi$ is at least its streak when $\pi\ne1\cdots k$, we can equivalently say that $\pi$ is grounded if it is not equal to $1\cdots k$ and its max height is equal to its streak. The heights of the elements in $\pi$ are then bounded above by its streak, hence the term grounded.

The description of the family of equivalences proceeds with a long series of lemmas. We first describe a decomposition of clusters for grounded $\pi$ into pseudo clusters.

\begin{defn}
\label{pcdecompdef}
Suppose that $\pi$ is a grounded pattern with streak $s$. Then for any $1\le p\le s$, any $p$-pseudo $m$-cluster $C$ with respect to $\pi$, we define the \emph{standard decomposition} of $C$ as $(P,\mathcal Q)$, where
\begin{enumerate}
    \item $P$ is the primitive $p$-pseudo cluster consisting of all highlighted instances of patterns that contain the root of $C$, and
    \item $\mathcal Q$ is the collection of the pseudo clusters $C_v$ over all non-root vertices $v$ of $P$. Here, we define $C_v$ to be the $h$-pseudo cluster rooted at $v$ consisting of the vertices in $C$ whose lowest ancestor in $P$ is $v$, where $H$ is the height of $v$ in $P$.
\end{enumerate}
\end{defn}

\begin{figure}[ht]
    \centering
    \scalebox{0.7}
    {\forestset{filled circle/.style={
      circle,
      text width=4pt,
      fill,
    },}
    \begin{forest}
    for tree={filled circle, inner sep = 0pt, outer sep = 0 pt, s sep = 1 cm}
    [, fill=white
        [, edge=white, edge label={node[above,black]{1}}, name=v1
            [, edge label={node[above,black]{2}}, name=v2
                [, edge=red, edge label={node[left,black]{4}}, name=v4
                    [, edge=red, edge label={node[left,black]{6}}, name=v6
                        [, edge=red, edge label={node[left,black]{5}}, name=v5
                        ]
                    ]
                    [, edge=yellow!60!red, edge label={node[left,black]{24}}, name=v24
                        [, edge=yellow!60!red, edge label={node[left,black]{26}}, name=v26
                            [, edge=yellow!60!red, edge label={node[left,black]{25}}, name=v25
                            ]
                        ]
                    ]
                ]
                [, edge label={node[right,black]{9}}, name=v9
                    [, edge label={node[left,black]{8}}, name=v8
                    ]
                    [, edge=blue, edge label={node[left,black]{19}}, name=v19
                        [, edge=blue, edge label={node[left,black]{23}}, name=v23
                            [, edge=blue, edge label={node[left,black]{22}}, name=v22
                            ]
                        ]
                    ]
                    [, edge=cyan, edge label={node[right,black]{21}}, name=v21
                        [, edge=cyan, edge label={node[right,black]{12}}, name=v12
                        ]
                    ]
                ]
            ]
            [, edge=black!60!white, edge label={node[right,black]{3}}, name=v3
                [, edge=black!60!white, edge label={node[right,black]{20}}, name=v20
                    [, edge=black!60!white, edge label={node[right,black]{13}}, name=v13
                    ]
                ]
            ]
            [, edge=black!30!white, edge label={node[right,black]{15}}, name=v15
                [, edge=black!30!white, edge label={node[right,black]{7}}, name=v7
                    [, edge=green, edge label={node[right,black]{10}}, name=v10
                        [, edge=green, edge label={node[right,black]{16}}, name=v16
                            [, edge=green, edge label={node[right,black]{14}}, name=v14
                            ]
                        ]
                    ]
                    [, edge=black!60!green, edge label={node[right,black]{11}}, name=v11
                        [, edge=black!60!green, edge label={node[right,black]{18}}, name=v18
                            [, edge=black!60!green, edge label={node[right,black]{17}}, name=v17
                            ]
                        ]
                    ]
                ]
            ]
        ]
    ]
    \path[cyan,dashed,out=0,in=90] (v2.child anchor) edge (v9.parent anchor);
    \end{forest}}
    \caption{The standard decomposition of a $2$-pseudo cluster with respect to $1243$. The highlighted instances on $\{1,2,9,8\}$, $\{1,3,20,13\}$, and $\{1,15,7\}$ in the primitive portion $P$ are colored with different shades of grey. The $C_v$ are colored in different shades of red, blue, and green. The red pseudo cluster consists of instances $2,4,6,5$ and $4,24,26,25$, the blue pseudo cluster consists of instances $9,19,23,22$ and $9,21,12$, and the green pseudo cluster consists of instances $7,10,16,14$ and $7,11,18,17$. Note that the vertex labeled $2$ is not in the blue pseudo cluster corresponding to the vertex labeled $9$ in the decomposition, but it is a part of the highlighted instance $2,9,21,12$ in the whole pseudo cluster.}
    \label{standecfig}
\end{figure}
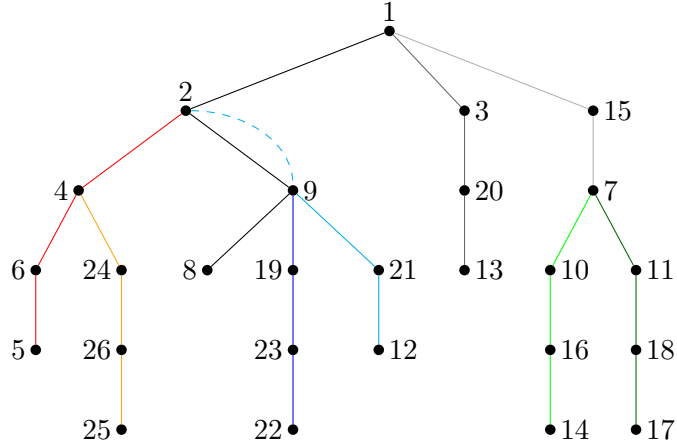

We will show the existence and uniqueness of the standard decomposition in Lemma \ref{pcdecomp}, along with some other properties. To do so, we first show an intermediate result.

\begin{lem}
\label{rootsmall}
Let $\pi$ be a grounded pattern with streak $s$. In any $p$-pseudo cluster with respect to $\pi$ with $p\le s$, the root has label $1$.
\end{lem}

\begin{proof}
Suppose that we have two intersecting instances $v_1,\ldots,v_k$ and $w_1,\ldots,w_k$ of $\pi$, where $v_i,w_i$ are labeled vertices and $v_1,\ldots,v_k$ and $w_1,\ldots,w_k$ are paths. Suppose that $v_i=w_j$ is the last vertex that the two instances share in common, i.e. the largest $i,j$ with $v_i=w_j$ so that $v_{i+1}\ne w_{j+1}$, if both $v_{i+1}$ and $w_{j+1}$ exist. Without loss of generality, suppose that $i\ge j$, which means that $w_1\in\{v_1,\ldots,v_k\}$. We assume that $i>j$, or equivalently that $v_1\ne w_1$. In this case, we must have that $j$ is at most the height $h$ of $\pi(i)$ in $\pi$. Indeed, note that we must have $j\le s$ because the first $s+1$ terms of $\pi$ are increasing. If $j>s$, then $w_1,\ldots,w_{s+1}$ would be a consecutive subsequence of $v_2,\ldots,v_k$ with increasing labels, so some $v_q$ has height at least $s+1$, a contradiction. Since the first $s$ terms of $\pi$ are increasing, $w_1,\ldots,w_j$ have increasing labels, so since $v_i=w_j$, $j$ is upper bounded by the height of $\pi(i)$ in $\pi$. Since $j\le s$ and the first $s$ terms of $\pi$ are $1,\ldots,s$, this means that the labels of $w_{j+1},\ldots,w_k$ are greater than the labels of $w_1,\ldots,w_j$, which are greater than the labels of $v_1,\ldots,v_{\min(s,i-1)}$ (which are in turn smaller than the labels of $v_{\min(s+1,i)},\ldots,v_k$). Thus, the smallest labels among the two instances are the labels of $v_1,\ldots,v_{\min(s,i-1)}$. In particular, this means that if we instead took an instance $v_p,\ldots,v_k$ of $\pi(p)\cdots\pi(k)$ for $p\le s$ and an instance $w_1,\ldots,w_k$ of $\pi$ such that $w_1\in\{v_p,\ldots,v_k\}$, then $v_p$ would have the smallest label among all vertices. Also, if we instead took an instance $v_{p_1},\ldots,v_k$ of $\pi(p_1)\cdots\pi(k)$ and an instance $w_{p_2},\ldots,w_k$ of $\pi(p_2)\cdots\pi(k)$ with $p_1,p_2\le s$ with $v_{p_1}=w_{p_2}$, then $v_{p_1}=w_{p_2}$ would have the smallest label among all vertices. In other words, whenever two possibly truncated (by no more than $s$ elements in the beginning) instances of $\pi$ intersect, where truncated instances must begin at the root of the tree formed by the two overlapping paths, the root will have the smallest label. This automatically implies that in any $p$-pseudo cluster with $p\le s$, the root has the smallest label.
\end{proof}

\begin{lem}
\label{pcdecomp}
Let $\pi$ be a grounded pattern with streak $s$, $C$ be a $p$-pseudo $m$-cluster with respect to $\pi$, and $r$ be the root vertex of $C$. Then the standard decomposition $(P,\mathcal Q)$ of $C$ exists and is unique. Furthermore, every vertex of $C$ lies in $P$ or a pseudo cluster in $\mathcal Q$, the pseudo clusters in $\mathcal Q$ are disjoint, and the only vertex that lies in both $P$ and $C_v$ is $v$ for all $v\in P\setminus\{r\}$. Every highlighted instance of a pattern in $C$ corresponds to a highlighted instance of a pattern in exactly one of $P$ and the pseudo clusters in $\mathcal Q$, and corresponding patterns end on the same vertex. The total number of highlighted instances of patterns throughout $P$ and the pseudo clusters in $\mathcal Q$ is thus $m$. Finally, in $P$ and the pseudo clusters in $\mathcal Q$, the root of the pseudo cluster has the smallest label.
\end{lem}

\begin{proof}
Note that $P$ is uniquely determined by its definition, as it is the cluster consisting of the highlighted instances of patterns that contain the root of $C$. For each non-root vertex $v$ of $P$, we consider the tree $T_v$ of elements whose lowest ancestor that lies in $P$ is $v$. In this way, all of the $T_v$ are disjoint, all vertices of $C$ lie in one of $P$ and $T_v$ for $v\notin P$, and the only vertex that lies in both $P$ and $T_v$ is $v$. The $T_v$ will serve as the underlying trees of the pseudo clusters $C_v$.

We now highlight the instances of patterns in $P$ and the $T_v$. Every highlighted instance of a pattern in $C$ is contained entirely in one of $P$ and the $T_v$ or contains vertices in both $P$ and $T_v$ for some $v$. If a highlighted instance of a pattern in $C$ is contained entirely in one of $P$ and the $T_v$, then we highlight it in the corresponding tree, whichever one of $P$ and the $T_v$ it is. Note that this is consistent with the way we defined $P$ to be the pseudo cluster of highlighted instances of patterns containing the root of $C$, as no highlighted instance of a pattern that does not contain the root of $C$ can be entirely contained in $P$. It is also clear that $P$ is a primitive $p$-pseudo cluster. For every other highlighted instance $u_1,\ldots,u_k$ of a pattern, which is necessarily a highlighted instance of $\pi$ since it does not contain the root of $C$, we consider the maximal index $j$ such that $u_j\in P$. By the same arguments as in the previous paragraph, we must have that $j$ is at most the height $h$ of $u_j$ in $P$. We then highlight the instance $u_j,\ldots,u_k$ of $\pi(j)\cdots\pi(k)$ in $T_{u_j}$. As $j\le h$, this does not violate the condition that $C_{u_j}$ is to be an $h$-pseudo cluster. Furthermore, it is also clear that the highlighted instances of patterns in $T_v$ are all connected, which means that they form a cluster $C_v$. In this way, all highlighted instances of patterns in $C$ have a corresponding highlighted instance of a pattern in one of $P$ and the pseudo clusters in $\mathcal Q$, and corresponding instances end on the same vertex. By Lemma \ref{rootsmall}, in all of $P$ and the pseudo clusters in $\mathcal Q$, the root has the smallest label, so all of the claimed conditions have been proven.

Finally, note that given any standard decomposition $(P,\mathcal Q)$, it is easy to recover the original pseudo cluster $C$. The underlying tree of $C$ is simply the union of the underlying trees of $P$ and the pseudo clusters in $\mathcal Q$. We keep any highlighted instances of patterns containing the root or highlighted instances of $\pi$. All other highlighted instances of patterns are instances $v_j,\ldots,v_k$ of $\pi(j)\cdots\pi(k)$ that begin at a vertex $v_j$ of $P$. In this case, we have by construction that $v_j$ is the only vertex in the instance that is in $P$ and that $j$ is at most the height $h$ of $v_j$ in $P$, as $C_{v_j}$ is an $h$-pseudo cluster. Thus, if $u_1,\ldots,u_h=v_j$ is the path in $P$ of increasing labels, we can highlight $u_{h-j+1},\ldots,u_h=v_j,\ldots,v_k$ as the corresponding highlighted instance of $\pi$ in $C$. Doing so for all highlighted instances of patterns in $P$ and the $T_v$ recovers the original pseudo cluster $C$. Consequently, $p$-pseudo $m$-clusters $C$ are in bijection with standard decompositions $(P,\mathcal Q)$.
\end{proof}

\begin{lem}
\label{ptostrong}
Suppose that $\pi$ and $\pi'$ are grounded patterns of length $k$ of the same streak $s$. If $\pi$ and $\pi'$ are $p$-equivalent for all $1\le p\le s$, then the number of $p$-pseudo $m$-clusters of size $n$ with respect to $\pi$ and $\pi'$ are equal for all $m,n,1\le p\le s$. In particular, by taking $p=1$ we have that $\pi$ and $\pi'$ have the same forest cluster numbers and are thus strongly equivalent by Theorem \ref{gpfclust}.
\end{lem}

\begin{proof}
Let $r_{m,n,p}$ be the number of $p$-pseudo $m$-clusters of $n$ vertices with respect to $\pi$, and let $r_{m,n,p}'$ denote the same with respect to $\pi'$. Our approach, mimicking the proof of Theorem \ref{gpfclust} that was given in \cite{GP}, is to show that $r_{m,n,p}$ and $r_{m,n,p}'$ satisfy the same recursion, or equivalently by induction on $n$. The base cases of $n=0$ and $n=1$ are clear, as for $n=0$ there is only the empty pseudo cluster and for $n=1$ we have that $r_{m,1,p}=r_{m,1,p}'=0$. For the recursion or the inductive step, we make use of the standard decomposition.

For $m,n\ge0,1\le p\le s,f:\{2,\ldots,n\}\rightarrow[s]$, let $B(m,n,p,f)$ denote the number of primitive $p$-pseudo $m$-clusters $C$ of size $n$ such that for all $1<i\le n$, the height of the vertex with label $i$ in $C$ is $f(i)$. Note here that we have shown before that such a pseudo cluster must have the smallest label at its root and that the height of every non-root vertex is at most $s$, so it suffices to take $f$ to be a function from $\{2,\ldots,n\}$ to $[s]$, as otherwise $B(m,n,p,f)=0$. We define $B'(m,n,p,f)$ analogously for $\pi'$ and note that $B(m,n,p,f)=B'(m,n,p,f)$ by assumption of the $p$-equivalence of $\pi$ and $\pi'$.

Suppose that $L_P$ is a subset of $\{2,\ldots,n\}$ with $|L_P|=q$ and $L_P=\{\ell_1,\ldots,\ell_q\}$ with $\ell_1<\cdots<\ell_q$. Let $L_1\sqcup\cdots\sqcup L_q$ be a partition of $\{2,\ldots,n\}$ such that $L_i\subseteq\{\ell_i,\ldots,n\}$ and $L_P\cap L_i=\{\ell_i\}$ for all $1\le i\le q$. For simplicity, let $\mathcal L_n$ be the set of all pairs $(L_P,(L_1,\ldots,L_q))$ satisfying these conditions. Given integers $m,n>0$, $1\le p\le s$, $(L_P,(L_1,\ldots,L_q))\in\mathcal L_n$, nonnnegative integers $m_0,\ldots,m_q$ summing to $m$, and a function $f:L_P\rightarrow[s]$, let $N(m,n,p,L_P,(L_1,\ldots,L_q),(m_0,\ldots,m_q),f)$ denote the number of $p$-pseudo $m$-clusters $C$ of size $n$ with respect to $\pi$ such that if the standard decomposition of $C$ is $(P,\mathcal Q)$, then \begin{itemize}
    \item $P$ has label set $L_P\cup\{1\}$ (the label $1$ is automatically the root, as all pseudo clusters have root label $1$), the height of the vertex labeled with $\ell_i$ in $P$ is $f(\ell_i)$, and there are exactly $m_0$ highlighted instances of a pattern completely in $P$, and
    \item a pseudo cluster $C_v$ in $\mathcal Q$ with root $v$ has label set $L_i$ if $v$ has label $\ell_i$ and there are exactly $m_i$ highlighted instances of a pattern in $C_v$,
\end{itemize}
and let $N'(m,n,p,L_P,(L_1,\ldots,L_q),(m_0,\ldots,m_q),f)$ denote the analogous quantity for $\pi'$.

By summing over all possible $L_p,(L_1,\ldots,L_q),(m_0,\ldots,m_q),f$, we have that \[r_{m,n,p}=\sum_{(L_P,(L_1,\ldots,L_q))\in\mathcal L_n}\sum_{m_0+\cdots+m_q=m}\sum_{f:L_P\rightarrow[s]}N(m,n,p,L_P,(L_1,\ldots,L_q),(m_0,\ldots,m_q),f).\]
On the other hand, we have that \[N(m,n,p,L_P,(L_1,\ldots,L_q),(m_0,\ldots,m_q),f)=B(m_0,q+1,p,f\circ g_{L_P})\prod_{|L_i|>1\text{ or }m_i>0}r_{m_i,|L_i|,f(i)}\] where $g_{L_P}:\{2,\ldots,q+1\}\rightarrow L_P$ is the order-preserving bijection between $\{2,\ldots,q+1\}$ and $L_P$. Indeed, by Lemma \ref{pcdecomp}, pseudo clusters $C$ satisfying the conditions in the bullets above are in bijection with the standard decompositions $(P,\mathcal Q)$ that satisfy the conditions. We may then independently decide the pseudo cluster $P$ and the pseudo clusters in $\mathcal Q$ and multiply the results together. For $P$, we want to count the number of primitive $p$-pseudo $m_0$-clusters of size $q+1$ with height function $f\circ g_{L_P}$ as only the relative order matters for counting pseudo clusters. For $C_v\in\mathcal Q$ with root $v$, when $|L_i|=1$ and $m_i=0$, this corresponds to not attaching any extra highlighted instances of patterns to $v$, of which there is only one way. Otherwise, this is by definition counted by $r_{m_i,|L_i|,f(i)}$, demonstrating the identity.

Consequently we have that
\[r_{m,n,p}=\sum_{(L_P,(L_1,\ldots,L_q))\in\mathcal L_n}\sum_{m_0+\cdots+m_q=m}\sum_{f:L_P\rightarrow[s]}B(m_0,q+1,p,f\circ g_{L_P})\prod_{|L_i|>1\text{ or }m_i>0}r_{m_i,|L_i|,f(i)}\]
as well as
\[r_{m,n,p}'=\sum_{(L_P,(L_1,\ldots,L_q))\in\mathcal L_n}\sum_{m_0+\cdots+m_q=m}\sum_{f:L_P\rightarrow[s]}B'(m_0,q+1,p,f\circ g_{L_P})\prod_{|L_i|>1\text{ or }m_i>0}r_{m_i,|L_i|,f(i)}'\]
by applying the same arguments to $\pi'$. We have assumed that $\pi$ and $\pi'$ are $p$-equivalent. Thus, $B(m_0,q+1,p,f\circ g_{L_P})=B'(m_0,q+1,p,f\circ g_{L_P})$ for all $m_0,q,1\le p\le s,f\circ g_{L_P}:\{2,\ldots,q+1\}\rightarrow[s]$, so $r_{m,n,p}$ and $r'_{m,n,p}$ satisfy the same recursion and the lemma follows.
\end{proof}

The next step is show a way to ``boost'' $1$-equivalence between certain (not necessarily grounded) patterns to $p$-equivalence between longer patterns. It allows us to more easily find examples of the stronger condition of $p$-equivalence.

\begin{lem}
\label{boost}
Suppose that $\pi$ and $\pi'$ are $1$-equivalent permutations of length $k$ with the property that $\pi(1)>\pi(2)$ and $\pi'(1)>\pi'(2)$. Then $\pi$ and $\pi'$ have the same max height $h$, and for all $\ell$, $\sigma=1,\ldots,\ell,\pi(1)+\ell,\ldots,\pi(k)+\ell$ and $\sigma'=1,\ldots,\ell,\pi'(1)+\ell,\ldots,\pi'(k)+\ell$ are $p$-equivalent for all $1\le p\le\ell$. By Lemma \ref{ptostrong}, if $\ell\ge h$, then $\sigma$ and $\sigma'$ are strongly equivalent.
\end{lem}

\begin{proof}
To see that $\pi$ and $\pi'$ have the same max height, we consider clusters on $k$ vertices with $1$ instance of $\pi$ or $\pi'$. These clusters must consist of a path with labels in the order of $\pi$ or $\pi'$. The $1$-equivalence of $\pi$ and $\pi'$ then implies that $\pi$ and $\pi'$ start with the same number $a$ and that for all $i\ne a$, the height of $i$ in $\pi$ is equal to the height of $i$ in $\pi'$, so $\pi$ and $\pi'$ have the same max height.

The $1$-equivalence of $\pi$ and $\pi'$ is equivalent to the existence of a root- and height-preserving bijection between primitive $m$-clusters of size $n$ for $\pi$ and primitive $m$-clusters of size $n$ for $\pi'$. Here, root-preserving means that the label of the root is fixed and height-preserving means that for each label different from the label of the root, the height of the vertex having that label in the tree is fixed (though the vertex that has that label and even the underlying unlabeled forest structure may change). We specify that such a bijection should be both root-preserving and height-preserving because the height is technically only defined for non-root vertices. Let $\alpha$ be a bijection mapping primitive $m$-clusters $C$ for $\pi$ to primitive $m$-clusters $C'$ for $\pi'$. By assumption, $\alpha$ has an inverse $\beta$, so $\alpha$ and $\beta$ are inverse maps between primitive clusters on $\pi$ and $\pi'$.

To prove the lemma, we will demonstrate a root- and height-preserving bijection between primitive $p$-pseudo $m$-clusters of size $n$ for $\sigma$ and primitive $p$-pseudo $m$-clusters of size $n$ for $\sigma'$ for all $1\le p\le\ell$ based on $\alpha$ and $\beta$. Before we explain the bijection, we first make some observations about primitive pseudo clusters for $\sigma$ and $\sigma'$. First, in any primitive pseudo cluster it is unnecessary to highlight instances of patterns, because the highlighted instances are always going to consist of the paths from the root to the leaves. In this way, highlighted instances correspond directly to the leaves of the underlying tree. Consider a primitive $p$-pseudo cluster $C$ for $\sigma=\sigma(1)\cdots\sigma(k+\ell)$ with $1\le p\le\ell$. For each instance $v_i,\ldots,v_{k+\ell}$ of $\sigma(i)\cdots\sigma(k+\ell)$, color $v_j$ with the color $j$ for all $i\le j\le k+\ell$, so vertices may receive multiple colors. Note that all vertices colored with a number greater than $\ell+1$ are only colored by that number. Indeed, suppose that such a vertex $v$ was colored by both $a$ and $b$, where $a>b$ and $a>\ell+1$. Then note that the parent of $v$ is colored by both $a-1$ and $b-1$. By taking parents and going up the tree, we may assume that $a=\ell+2$. We never end up at the root during this process since $p\le\ell$, so the root is only colored with numbers at most $\ell$. Note that $\sigma(1)<\cdots<\sigma(\ell+1)>\sigma(\ell+2)$. Thus, if $v$ is colored with $\ell+2$ then its label is smaller than its parent's label. But if $v$ is colored with $b\le\ell+1$ then its label is greater than its parent's label, a contradiction.

Define the \emph{fixed tree} $T$ of $C$ as the subtree of vertices whose colors are all at most $\ell+1$. Let the \emph{seeds} of $T$ be the vertices that are labeled $\ell+1$. Note that the root cannot be a seed as $p\le\ell$. By considering the colors of the vertices, for each vertex $u$ not in $T$, the lowest ancestor of $u$ in $T$ is a seed. For each seed $v$, we let $T_v$ be the subtree of $C$ consisting of $v$ and the vertices outside of $T$ for which $v$ is the lowest ancestor in $T$, which necessarily contains vertices other than $v$. Then, $C$ is the union of $T$ and $T_v$ for seeds $v\in T$, where the $T_v$ are disjoint, and the only vertices in both $T$ and $T_v$ is $v$. Let $\mathcal U$ be the collection of $T_v$ over all seeds $v\in T$, and call the decomposition of $C$ into $(T,\mathcal U)$ the \emph{primitive decomposition} of $C$. Note that it is well-defined for both $\sigma$ and $\sigma'$, as $\sigma'$ also satisfies all of the properties that we used throughout this paragraph.

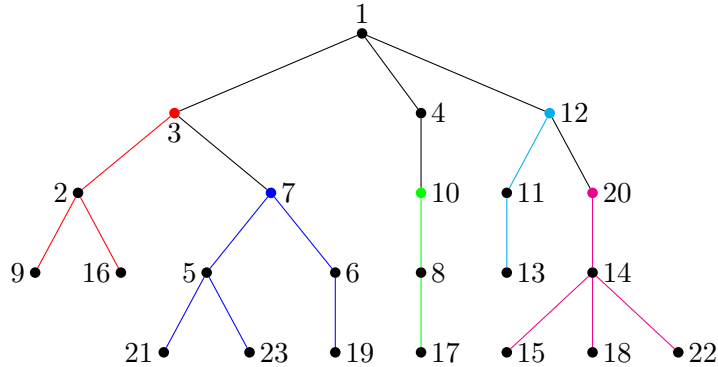
\begin{figure}[ht]
    \centering
    \scalebox{0.7}
    {\forestset{filled circle/.style={
      circle,
      text width=4pt,
      fill,
    },}
    \begin{forest}
    for tree={filled circle, inner sep = 0pt, outer sep = 0 pt, s sep = 1 cm}
    [, fill=white
        [, edge=white, edge label={node[above,black]{1}}
            [, fill=red, edge label={node[black,below]{3}}
                [, edge=red, edge label={node[black, above,left]{2}}
                    [, edge=red, edge label={node[black, above,left]{9}}
                    ]
                    [, edge=red, edge label={node[black, above,left]{16}}
                    ]
                ]
                [, fill=blue, edge label={node[black,right]{7}}
                    [, edge=blue, edge label={node[left,black]{5}}
                        [, edge=blue, edge label={node[left,black]{21}}
                        ]
                        [, edge=blue, edge label={node[right,black]{23}}
                        ]
                    ]
                    [, edge=blue, edge label={node[right,black]{6}}
                        [, edge=blue, edge label={node[right,black]{19}}
                        ]
                    ]
                ]
            ]
            [, edge label={node[right,black]{4}}
                [, fill=green, edge label={node[right,black]{10}}
                    [, edge=green, edge label={node[right,black]{8}}
                        [, edge=green, edge label={node[right,black]{17}}
                        ]
                    ]
                ]
            ]
            [, fill=cyan, edge label={node[right,black]{12}}
                [, edge=cyan, edge label={node[right,black]{11}}
                    [, edge=cyan, edge label={node[right,black]{13}}
                    ]
                ]
                [, fill=magenta, edge label={node[right,black]{20}}
                    [, edge=magenta, edge label={node[right,black]{14}}
                        [, edge=magenta, edge label={node[right,black]{15}}
                        ]
                        [, edge=magenta, edge label={node[right,black]{18}}
                        ]
                        [, edge=magenta, edge label={node[right,black]{22}}
                        ]
                    ]
                ]
            ]
        ]
    ]
    \end{forest}}
    \caption{The primitive decomposition of a pseudo cluster with respect to $12435$. The fixed tree consists of the vertices labeled with $1,3,4,7,10,12,20$. We have not highlighted pattern instances as they are just the paths from the root to the leaves. Instead, we have colored the primitive $435$-clusters and their seeds. The red cluster is rooted at $3$ and has instances on $\{3,2,9\}$ and $\{3,2,16\}$, the blue cluster is rooted at $7$ and has instances on $\{7,5,21\}$, $\{7,5,23\}$, and $\{7,6,19\}$, the green cluster is rooted at $10$ and has an instance on $\{10,8,17\}$, the cyan cluster is rooted at $12$ and has an instance on $\{12,11,13\}$, and the magenta cluster is rooted at $20$ and has instances on $\{20,14,15\}$, $\{20,14,18\}$, and $\{20,14,22\}$.}
    \label{seedfig}
\end{figure}

The last observation we need is that for a primitive $p$-pseudo cluster $C$ for $\sigma$ with $p\le\ell$ with primitive decomposition $(T,\mathcal U)$, all of the clusters $T_v$ are primitive clusters for $\pi$. This is clear from the definition of the primitive decomposition. Now, given a primitive $p$-pseudo $m$-cluster $C$ of size $n$ for $\sigma$, we consider its primitive decomposition $(T,\mathcal U)$. For each $T_v\in\mathcal U$, we apply the map $\alpha$ (using the relative order of the labels in $T_v$). The result is a primitive $p$-pseudo $m$-cluster $C'$ of size $n$ for $\sigma'$. This is as for each seed $v$, the labels of $v$'s strict ancestors are all less than the labels of $v$'s descendants. By replacing each root-to-leaf instance of $\pi$ in $T_v$ with a root-to-leaf instance of $\pi'$, we replace each root-to-leaf instance of $i,\ldots,\ell,\pi(1)+\ell,\ldots,\pi(k)+\ell$ in $C$ with a root-to-leaf instance of $i,\ldots,\ell,\pi'(1)+\ell,\ldots,\pi'(k)+\ell$. Since $\alpha$ fixes the root and the number of vertices and number of leaves by definition, the result is a $p$-pseudo $m$-cluster $C'$ of size $n$ for $\sigma'$. Furthermore, it is clear that $C'$ has the same fixed tree (including labels) and seeds as $C$. Thus, if $(T,\mathcal U')$ is the primitive decomposition of $C'$ where $\mathcal U'=\bigcup_{v\text{ seed}}T_v'$, then applying $\beta$ to each $T_v'$ transforms them back into $T_v$, giving the original cluster $C$. Thus, we have defined a bijection between pseudo clusters for $\sigma$ and pseudo clusters for $\sigma'$. It remains to show that this bijection fixes the root and height of each vertex. The root is clearly fixed since it is a part of the fixed tree. Vertices and labels in the fixed tree are fixed, so their heights are as well. To finish, we note that for a vertex $w\ne v$ in $T_v$ where $v$ is a seed of the fixed tree $T$ of $C$, the height of $w$ in $C$ is equal to the height of $w$ in $T_v$. This is as $\sigma(\ell+1)>\sigma(\ell+2)$ due to $\pi(1)>\pi(2)$, so all of the children of a seed $v$ in $T_v$ have smaller labels than $v$. Thus, as $\alpha$ and $\beta$ preserve the heights, so does the bijection we have defined. Finally, note that $\sigma$ and $\sigma'$ are grounded patterns of streak $\ell$, so by Lemma \ref{ptostrong} they are strongly equivalent.
\end{proof}

The next step is to obtain $1$-equivalences from primitive structure equivalences. The rough idea is that if $x,\ldots,y$ all have the same height in a pattern $\pi$, are not consecutive, and appear after $x-1$ and $y+1$, then we are able to permute the vertices in the layers of any primitive cluster corresponding to the positions of $x,\ldots,y$ in $\pi$ following primitive structure equivalence. This follows because the appearance of $x-1$ and $y+1$ before $x,\ldots,y$ ``pin down'' the values of those layers so that they are able to be freely permuted without affecting the relative order with the rest of the forest.

\begin{lem}
\label{swaplem}
Let $\pi$ be a permutation of length $k$, and consider indices $1<a_1<\ldots<a_i\le k$ with $a_{j+1}-a_j>1$ for all $1\le j<i$ \emph{(}i.e. the indices are not consecutive\emph{)} such that the heights of $\pi(a_j)$ in $\pi$ are equal for all $j$. Suppose that
\begin{itemize}
    \item $\{\pi(a_1),\ldots,\pi(a_i)\}=\{x,\ldots,y\}$ for some $x<y$ and that
    \item $x-1$ and $y+1$ \emph{(}if they are in $[n]$\emph{)} appear before $\pi(a_1)$ in $\pi$.
\end{itemize}
Let $\sigma$ and $\sigma'$ be permutations of length $i$ that are primitive structure equivalent. Suppose that $\sigma$ and $\pi(a_1),\ldots,\pi(a_i)$ are in the same relative order, and let $\pi'(a_1),\ldots,\pi'(a_i)$ be the permutation of $\pi(a_1),\ldots,\pi(a_i)$ that is in the same relative order as $\sigma'$. Then $\pi$ is $1$-equivalent to $\pi'$, where $\pi'$ is $\pi$ with the subsequence $\pi(a_1),\ldots,\pi(a_i)$ is replaced by $\pi'(a_1),\ldots,\pi'(a_i)$, i.e. define $\pi'(a_j)$ as before and $\pi'(j)=\pi(j)$ for $j\notin\{a_1,\ldots,a_i\}$.
\end{lem}

\begin{proof}
We will be defining a root- and height-preserving bijection between primitive clusters for $\pi$ and $\pi'$. In fact, our bijection will be structure-preserving, so it will only permute certain non-root labels, specifically the labels of the vertices with depth in $\{a_1,\ldots,a_i\}$. Note that the conditions in the lemma statement are symmetric with respect to $\pi$ and $\pi'$, i.e. they are satisfied for $\pi$ if and only if they are satisfied for $\pi'$ (this holds for any permutation $\pi'$ that starts with $\pi$ and permutes some of $\pi(a_1),\ldots,\pi(a_i)$). This is because the elements of $\pi$ outside of $\pi(a_1),\ldots,\pi(a_i)$ are either greater than all of $\pi(a_1),\ldots,\pi(a_i)$ or less than all of $\pi(a_1),\ldots,\pi(a_i)$. None of the elements $\pi(a_1),\ldots,\pi(a_i)$ are adjacent, so permuting the $\pi(a_j)$ does not impact the equal heights condition. It in fact preserves the heights of all $\pi(j)$ in $\pi$. All other conditions are defined symmetrically so are always satisfied when $\pi(a_1),\ldots,\pi(a_i)$ are permuted.

Note that by the conditions in the statement, any structure-preserving bijection between primitive clusters for $\pi$ and $\pi'$ of the form described in the previous paragraph (where only labels of vertices with depth in $\{a_1,\ldots,a_i\}$ are permuted) will preserve the height of all vertices. Let $\varphi$ be the structure-preserving bijection from ordered primitive clusters for $\sigma$ to ordered primitive clusters for $\sigma'$. For each vertex $v$ of depth $a_1$ in a primitive cluster $C$ of $\pi$, we will permute labels within the subtree of $C$ rooted at $v$. The conditions in the lemma statement guarantee that if $S$ is the set of all of the descendants of $v$ that have depth in $\{a_1,\ldots,a_i\}$, then any descendant or ancestor $w$ of $v$ will either be in $S$, have label either greater than all of the labels of vertices in $S$, or have label less than all of the labels of vertices in $S$. Indeed, suppose that the label of $w\notin S$ is greater than the label of $v$. Consider the ancestor $u$ of $v$ of depth $j$, where $\pi(j)=y+1$, which exists by the conditions in the statement. We know that the label of $w$ is greater than the label of $u$. On the other hand, the label of $u$ is greater than all of the labels in $S$. Thus, the label of $w$ is greater than all of the labels in $S$. The case in which the label of $w$ is less than the label of $v$ is analogous. Thus, if the labels of the vertices in $S$ are permuted so that every instance of $\sigma$ (here we are looking at the vertices of depth $a_1,\ldots,a_i$ in a path starting at $v$) is replaced with an instance of $\sigma'$, then every instance of $\pi$ with the $a_1$st vertex at $v$ will be replaced with an instance of $\pi'$. If we do so for all $v$ of depth $a_1$, we will have replaced all instances of $\pi$ with an instance of $\pi'$. 

If $a_i<k$, then this replacement is precisely given by $\varphi$. Indeed, if we connect every vertex $u_1$ and $u_2$ in $S$ of depths $a_j$ and $a_{j+1}$ such that $u_1$ is an ancestor of $u_2$ for all $1\le j<i$ (i.e. the contraction of the tree to $S$), then we give $S$ the structure of an ordered primitive cluster for $\sigma$. This is because every path from a root to a leaf in this contraction is an instance of $\sigma$, but the children of each vertex are distinguished and their order matters. In the original graph, non-leaf vertices of $S$ have children that are not in $S$. Such children that are not in $S$ distinguish non-leaf vertices in from their siblings in $S$, so we indeed have an ordered primitive cluster for $\sigma$. Applying $\varphi$ and $\varphi^{-1}$ to each subtree rooted at depth $a_1$ gives us a desired bijection between primitive clusters for $\pi$ and primitive clusters for $\pi'$.

If $a_i=k$, then we need to modify the previous argument to account for the fact that the leaves of the contraction to $S$ do not have children in the original tree and may not have a distinguished order. The order of leaves $v_1,\ldots,v_m$ in $S$ does not matter if and only if $v_1,\ldots,v_m$ have a the same parent in the original tree. Thus, given an unlabeled tree structure, there is a fixed number $M$ that counts the number of ways to permute the labels of the leaves of a cluster that do not change the cluster. If we fix an unlabeled tree structure and enrich clusters for $\pi$ and $\pi'$ with one of the $M$ orders for their leaves, then the bijection from the previous paragraph directly applies. Thus, given a fixed unlabeled tree structure, the order-enriched clusters for $\pi$ and $\pi'$ are in bijection. This means that the number of clusters for $\pi$ with a given tree structure is equal to the number of clusters for $\pi'$ with a given tree structure because the fixed number $M$ of orders on the leaves is the same for $\pi$ and $\pi'$. Summing over all possible unlabeled tree structures gives the result.
\end{proof}

\begin{exmp}
\label{swapex}
Figure \ref{swapfig} gives an example of the bijection in the proof of Lemma \ref{swaplem}. In the pattern $1253764$, $3$ and $4$ have the same height of $1$, and $2$ and $5$ appear before them. Thus, we may swap them to get that $1253764$ and $1254763$ are $1$-equivalent, because the patterns $12$ and $21$ are primitive structure equivalent. A structure-preserving bijection between ordered primitive clusters for $12$ and $21$ is given by swapping the largest and smallest labels. In the primitive cluster for $1253764$, the vertices of depth $4$ and $7$, corresponding to the positions of $3$ and $4$ in the pattern, form ordered primitive clusters for $12$, which are colored red and blue. When the bijection between ordered primitive clusters for $12$ and $21$ is applied, we obtain a corresponding primitive cluster for $1254763$. Note that technically, the order of the vertices labeled $6$ and $7$ does not matter. If we enrich the cluster with an order between the vertices labeled $6$ and $7$, as well as an order between the vertices labeled $10$ and $11$, then each cluster for $1253764$ is counted $M=4$ times. Since the corresponding cluster for $1254763$ has the same structure, each cluster for $1253764$ of that structure is also counted $M=4$ times.
\end{exmp}

\begin{figure}[ht]
    \centering
    \scalebox{0.7}
    {\forestset{filled circle/.style={
      circle,
      text width=4pt,
      fill,
    },}
    \begin{minipage}[b]{0.45\linewidth}
    \centering
    \begin{forest}
    for tree={filled circle, inner sep = 0pt, outer sep = 0 pt, s sep = 1 cm}
    [, fill=white
        [, edge=white, edge label={node[above,black]{1}}
            [, edge label={node[left,black]{2}}
                [, edge label={node[left,black]{12}}
                    [, fill=red, edge label={node[right,black]{4}}, name=v4
                        [, edge label={node[left,black]{18}}
                            [, edge label={node[left,black]{14}}
                                [, fill=red, edge label={node[left,black]{6}}, name=v6
                                ]
                                [, fill=red, edge label={node[left,black]{7}}, name=v7
                                ]
                            ]
                        ]
                        [, edge label={node[left,black]{19}}
                            [, edge label={node[left,black]{15}}
                                [, fill=red, edge label={node[left,black]{8}}, name=v8
                                ]
                            ]
                        ]
                    ]
                ]
            ]
            [, edge label={node[right,black]{3}}
                [, edge label={node[right,black]{13}}
                    [, fill=blue, edge label={node[left,black]{5}}, name=v5
                        [, edge label={node[right,black]{20}}
                            [, edge label={node[right,black]{16}}
                                [, fill=blue, edge label={node[right,black]{9}}, name=v9
                                ]
                            ]
                        ]
                        [, edge label={node[right,black]{21}}
                            [, edge label={node[right,black]{17}}
                                [, fill=blue, edge label={node[right,black]{10}}, name=v10
                                ]
                                [, fill=blue, edge label={node[right,black]{11}}, name=v11
                                ]
                            ]
                        ]
                    ]
                ]
            ]
        ]
    ]
    \path[red,dashed,out=180,in=135] (v4.child anchor) edge (v6.parent anchor);
    \path[red,dashed,out=-90,in=135] (v4.child anchor) edge (v8.parent anchor);
    \path[red,dashed] (v4.child anchor) edge (v7.parent anchor);
    \path[blue,dashed,out=0,in=45] (v5.child anchor) edge (v11.parent anchor);
    \path[blue,dashed,out=-90,in=45] (v5.child anchor) edge (v9.parent anchor);
    \path[blue,dashed] (v5.child anchor) edge (v10.parent anchor);
    \end{forest}
    \end{minipage}
    \begin{minipage}[b]{0.45\linewidth}
    \centering
    \begin{forest}
    for tree={filled circle, inner sep = 0pt, outer sep = 0 pt, s sep = 1 cm}
    [, fill=white
        [, edge=white, edge label={node[above,black]{1}}
            [, edge label={node[left,black]{2}}
                [, edge label={node[left,black]{12}}
                    [, fill=red, edge label={node[right,black]{8}}, name=v4
                        [, edge label={node[left,black]{18}}
                            [, edge label={node[left,black]{14}}
                                [, fill=red, edge label={node[left,black]{6}}, name=v6
                                ]
                                [, fill=red, edge label={node[left,black]{7}}, name=v7
                                ]
                            ]
                        ]
                        [, edge label={node[left,black]{19}}
                            [, edge label={node[left,black]{15}}
                                [, fill=red, edge label={node[left,black]{4}}, name=v8
                                ]
                            ]
                        ]
                    ]
                ]
            ]
            [, edge label={node[right,black]{3}}
                [, edge label={node[right,black]{13}}
                    [, fill=blue, edge label={node[left,black]{11}}, name=v5
                        [, edge label={node[right,black]{20}}
                            [, edge label={node[right,black]{16}}
                                [, fill=blue, edge label={node[right,black]{9}}, name=v9
                                ]
                            ]
                        ]
                        [, edge label={node[right,black]{21}}
                            [, edge label={node[right,black]{17}}
                                [, fill=blue, edge label={node[right,black]{10}}, name=v10
                                ]
                                [, fill=blue, edge label={node[right,black]{5}}, name=v11
                                ]
                            ]
                        ]
                    ]
                ]
            ]
        ]
    ]
    \path[red,dashed,out=180,in=135] (v4.child anchor) edge (v6.parent anchor);
    \path[red,dashed,out=-90,in=135] (v4.child anchor) edge (v8.parent anchor);
    \path[red,dashed] (v4.child anchor) edge (v7.parent anchor);
    \path[blue,dashed,out=0,in=45] (v5.child anchor) edge (v11.parent anchor);
    \path[blue,dashed,out=-90,in=45] (v5.child anchor) edge (v9.parent anchor);
    \path[blue,dashed] (v5.child anchor) edge (v10.parent anchor);
    \end{forest}
    \end{minipage}}
    \caption{The patterns $1253764$ and $1254763$ are $1$-equivalent because the patterns $12$ and $21$ are primitive structure equivalent.}
    \label{swapfig}
\end{figure}
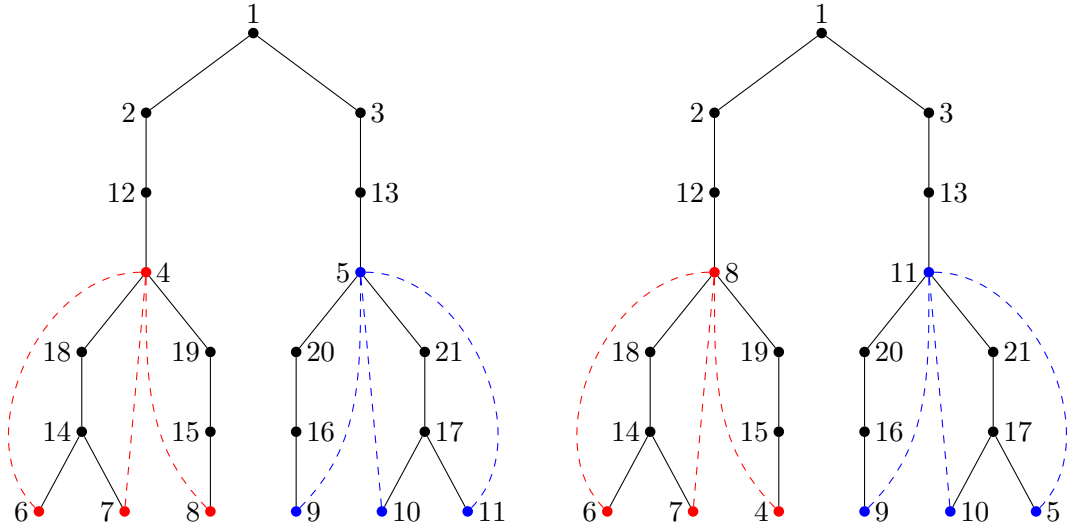

The final step in our construction is to find examples of primitive cluster equivalence. Complementation and appending $k+1$ give some primitive cluster equivalences which can be explicitly described, but there is also a family of recursively defined equivalences constructed in the same way as in Lemma \ref{swaplem}.

\begin{lem}
\label{pselem}
A permutation $\pi$ is primitive structure equivalent to its complement. If $\pi$ and $\pi'$ are permutations of length $k$ that are primitive structure equivalent, then $k+1,\pi$ and $k+1,\pi'$ are primitive structure equivalent. Finally, consider indices $1<a_1<\ldots,a_i\le k$ with the property that $\{\pi(a_1),\ldots,\pi(a_i)\}=\{x,\ldots,y\}$ for some $x<y$ and such that $x-1$ and $y+1$ \emph{(}if they are in $[n]$\emph{)} appear before $\pi(a_1)$ in $\pi$. Let $\sigma$ and $\sigma'$ be permutations of length $i$ that are primitive structure equivalent. Suppose that $\sigma$ and $\pi(a_1),\ldots,\pi(a_i)$ are in the same relative order, and let $\pi'(a_1),\ldots,\pi'(a_i)$ be the permutation of $\pi(a_1),\ldots,\pi(a_i)$ that is in the same relative order as $\sigma'$. Then $\pi$ and $\pi'$ are primitive structure equivalent, where $\pi'$ is $\pi$ but the subsequence $\pi(a_1),\ldots,\pi(a_i)$ is replaced by $\pi'(a_1),\ldots,\pi'(a_i)$, i.e. define $\pi'(a_j)$ as before and $\pi'(j)=\pi(j)$ for $j\notin\{a_1,\ldots,a_i\}$.
\end{lem}

\begin{proof}
The first part follows trivially by complementation of the labels of each ordered primitive cluster. For the second part, suppose we have a structure-preserving bijection $\varphi$ between ordered primitive clusters for $\pi$ and ordered primitive clusters for $\pi'$. Any ordered primitive cluster $C$ for $k+1,\pi$ or $k+1,\pi$ must have the largest label at the root. We can apply $\varphi$ to every subtree rooted at a child of the root of $C$, which is an ordered primitive cluster for $\pi$. This gives a structure-preserving bijection between ordered primitive clusters for $k+1,\pi$ and $k+1,\pi'$, as desired. The proof of third part is essentially the same as the proof of Lemma 3.17. We permute the labels of vertices of depth in $
\{a_1,\ldots,a_i\}$, where each subtree of each vertex of depth $a_1$ is permuted separately. Each subtree is permuted by considering the contraction of the tree to the vertices of depth $\{a_1,\ldots,a_i\}$ and applying the structure-preserving bijection given by the primitive structure equivalence of $\sigma$ and $\sigma'$. In contrast with the previous part, we do not require that the indices are not consecutive, because we do not need the vertex heights to be preserved in this case. Since we are already working with ordered primitive clusters, the vertices of depths in $\{a_1,\ldots,a_i\}$ already have the structure of an ordered rooted forest, and so there is no issue when permuting vertices of consecutive depths.
\end{proof}

Lemmas \ref{ptostrong}, \ref{boost}, and \ref{swaplem} reduce the problem of finding strong equivalences to the problem of finding $p$-equivalences for all $1\le p\le s$ to the problem of finding $1$-equivalences to the problem of finding primitive structure equivalences, many of which are described in Lemma \ref{pselem}. These primitive structure equivalences can be passed back up through the chain of lemmas to prove many strong equivalences. It is possible that there are permutations that satisfy the conditions of each lemma different from the ones given by the next lemma, but we are not aware of such methods. Doing so successfully would yield new strong equivalences.

\begin{exmp}
\label{cfwex}
Searching through all permutations of length $k\le5$ for $1$-equivalences given by Lemmas \ref{boost}, \ref{swaplem}, and \ref{pselem}, we find the following $1$-equivalences:
\begin{itemize}
    \item 3142, 3241
    \item 31425, 31524, 32415, 32514
    \item 31452, 32451
    \item 31542, 32541
    \item 52314, 52413
    \item 43152, 43251
    \item 53142, 53241
\end{itemize}
These give rise to the following infinite families of equivalences by Lemma 3.16:
\begin{itemize}
    \item $125364\stackrel c\sim125463,1236475\stackrel c\sim1236574,\ldots$
    \item $1253647\stackrel c\sim1253746\stackrel c\sim1254637\stackrel c\sim1254736,$
    
    $12364758\stackrel c\sim12364857\stackrel c\sim12365748\stackrel c\sim12365847,\ldots$
    \item $12364785\stackrel c\sim12365784,123475896\stackrel c\sim123476895,\ldots$
    \item $1253764\stackrel c\sim1254763,12364875\stackrel c\sim12365874,\ldots$
    \item $1274536\stackrel c\sim1274635,12385647\stackrel c\sim12385746,\ldots$
    \item $1265374\stackrel c\sim1265473,12376485
    \stackrel c\sim12376584,\ldots$
    \item $1275364\stackrel c\sim1275463,12386475\stackrel c\sim12386574,\ldots$
\end{itemize}
These are in fact strong equivalences, but we have written $\stackrel c\sim$ instead of $\stackrel{sc}\sim$. Notably, many of these pairs are not c-Wilf equivalent in terms of permutations, specifically the non-overlapping permutations with different ending terms due to \cite[Theorem 1.9]{DE}. It seems difficult to describe all $1$-equivalences given by the Lemma \ref{swaplem}, but in principle all of the $1$-equivalences between permutations of length $k$ given by Lemma \ref{swaplem} can can be computed by checking all of the permutations of length $k$.
\end{exmp}

With this family of equivalences, we are now able to prove Theorems \ref{cequiv} and \ref{cclass}.

\begin{proof}[Proof of Theorem \ref{cclass}]
We now construct a c-forest-Wilf equivalence class of size at least $2^{n-4}$ among patterns of length $n$ for $n\ge6$. The cases of $n\le5$ follow trivially by complementation. We first give a large class of $1$-equivalences. Note that in a classically $\{213,231\}$-avoiding permutation, every term is either greater than or less than all of the terms following it. If $n=2k+1$ is odd, we consider classically $\{213,231\}$-avoiding permutations $a_1,\ldots,a_{k-1}$ and $b_1,\ldots,b_{k-1}$ of $1,\ldots,k-1$ and $k+1,\ldots,2k-1$. Then note that in $k,a_1,b_1,\ldots,a_{k-1},b_{k-1}$ the heights of $a_1,b_1,\ldots,a_{k-1},b_{k-1}$ are $1,2,\ldots,1,2$ so by Lemma \ref{swaplem}, $k,a_1,b_1,\ldots,a_{k-1},b_{k-1}$ over all classically $\{213,231\}$-avoiding permutations $a_1,\ldots,a_{k-1}$ and $b_1,\ldots,b_{k-1}$ are all $1$-equivalent. Thus, over all classically $\{213,231\}$-avoiding permutations $a_1,\ldots,a_{k-1}$ and $b_1,\ldots,b_{k-1}$ of $1,\ldots,k-1$ and $k+1,\ldots,2k-1$, the patterns $1,2,k+2,a_1+2,b_1+2,\ldots,a_{k-1}+2,b_{k-1}+2$ all lie in the same c-forest-Wilf equivalence class by Lemma \ref{boost}. There are $2^{k-2}\cdot2^{k-2}=2^{2k-4}=2^{n-5}$ such permutations. If $n=2k$ is even, we consider classically $\{213,231\}$-avoiding permutations $a_1,\ldots,a_{k-1}$ and $b_1,\ldots,b_{k-2}$ of $1,\ldots,k-1$ and $k+1,\ldots,2k-2$. By the same argument as before, $k,a_1,b_1,\ldots,a_{k-2},b_{k-2},a_{k-1}$ over all classically $\{213,231\}$-avoiding permutations $a_1,\ldots,a_{k-1}$ and $b_1,\ldots,b_{k-2}$ are all $1$-equivalent. Thus the patterns $1,2,k+2,a_1+2,b_1+2,\ldots,a_{k-2}+2,b_{k-2}+2,a_{k-1}+2$ over all $\{213,231\}$-avoiding permutations $a_1,\ldots,a_{k-1}$ and $b_1,\ldots,b_{k-2}$ of $1,\ldots,k-1$ and $k+1,\ldots,2k-2$ all lie in the same c-forest-Wilf equivalence class. There are $2^{k-2}\cdot2^{k-3}=2^{2k-5}=2^{n-5}$ such permutations. In both cases, we found $2^{n-5}$ such permutations, and taking complements gives a c-forest-Wilf equivalence class of size at least $2^{n-4}$. Note that for $n=6$ and $n=7$ we get the equivalences from the first two examples in Example \ref{cfwex}.
\end{proof}

Notably in the case when $n$ is even, all of the permutations are non-overlapping so we actually get permutations from $\frac{n-2}2$ different c-Wilf equivalence classes for permutations in the same equivalence class by \cite[Theorem 1.9]{DE}. In contrast to classical forest-Wilf equivalence, beyond trivial symmetries c-forest-Wilf equivalence seems to be largely unrelated to c-Wilf equivalences for permutations.

With our construction, we are now able to give the proof of Theorem \ref{cequiv}. We first need the following lemmas.

\begin{lem}
\label{lotsconsavoid}
For all $c<1$, there exists a positive integer $k$ such that for sufficiently large $n$, at least $c^nn!$ permutations of $[n]$ avoid the consecutive pattern $1,\ldots,k$.
\end{lem}

\begin{proof}
Divide a uniform random permutation of $[n]$ into consecutive blocks of $k$ numbers, where the last block may have less than $k$ numbers. If each block of $k$ numbers (so the last block is excluded if it has less than $k$ numbers) is not increasing, then the whole permutation avoids the consecutive pattern $1,\ldots,2k$. The probability that this occurs is $\left(1-\frac1{k!}\right)^{\lfloor n/k\rfloor}\ge\left(1-\frac1{k!}\right)^{n/k}$. Since $\left(1-\frac1{k!}\right)^{1/k}\rightarrow1$ as $k\rightarrow\infty$, the lemma follows.
\end{proof}

In order to make use of our construction, we need to find many $1$-equivalences using Lemma \ref{swaplem}. The easiest way to do so is if the numbers $3,1,2$ appear in that order and are not next to each other. In that case, we can then switch $1$ and $2$. A negligible number of permutations have two of $1,2,3$ next to each other, even when we restrict to $1\cdots k$ (consecutively) avoiding permutations. A third of the remaining permutations have $3$ before $1$ and $2$, giving us enough permutations to work with. The next lemma formalizes this.

\begin{lem}
\label{lotscanswap}
For all $c<1$, there exists a positive integer $k$ such that for sufficiently large $n$, at least $c^nn!$ permutations of $[n]$ avoid the consecutive pattern $1,\ldots,k$ and satisfy the additional property that $1,2,3$ are not consecutive in the permutation and $3$ appears before $1$ and $2$.
\end{lem}

\begin{proof}
Let $a_{m,k}$ denote the number of permutations of $[n]$ that avoid the consecutive pattern $1\cdots k$. By a theorem of Ehrenborg, Kitaev, and Perry in \cite{EKP}, there exist positive constants $b_k,c_k>d_k$ with $\frac{a_{n,k}}{n!}=b_kc_k^n+O(d_k^n)$. By Lemma \ref{lotsconsavoid}, we have that $c_k\rightarrow1$ as $k\rightarrow\infty$.

We first show that there are at most $6a_{n-1,k}$ permutations on $[n]$ that avoid the consecutive pattern $1\cdots k$ but have two of $1,2,3$ next to each other. Indeed, note that if $1$ and $2$ are consecutive, then by deleting $2$ and subtracting $1$ from all other numbers, we obtain a permutation on $[n-1]$ that avoids the consecutive pattern $1\cdots k$. For each permutation $\pi$ on $[n-1]$ that avoids the consecutive pattern $1\cdots k$, there are at most two such permutations on $[n]$ in which we could have deleted $2$ to obtain $\pi$. To recover the original permutations, we just insert $2$ in one of the two positions next to $1$ and increment all other numbers by $1$. If $2$ and $3$ are consecutive, we do the same thing but delete $3$ and subtract $1$ from all other numbers greater than $2$, resulting in at most $2a_{n-1,k}$ other permutations. If $1$ and $3$ are consecutive but not adjacent to $2$, then we do the same thing but delete $3$ and subtract $1$ from all other numbers greater than $3$, resulting in at most $2a_{n-1,k}$ other permutations. Here, we are using the fact that when $2$ is not adjacent to $1$ and $3$, we do not create any instances of $1\cdots k$ when we delete $3$. This covers all cases, so there are at most $6a_{n-1,k}$ such permutations in total.

To finish, note that in a permutation that avoids the consecutive pattern $1\cdots k$ such that no two of $1,2,3$ are consecutive, we are free to permute $1,2,3$ without affecting this property. Thus, in exactly a third of these permutations does $3$ appear before $1$ and $2$. It follows that the number of permutations with the desired property is at least $\frac{a_{n,k}-6a_{n-1,k}}3=b_k\left(\frac13-\frac2{nc_k}\right)c_k^nn!+O(d_k^nn!)$. Since $c_k>d_k$ and $c_k\rightarrow1$ as $k\rightarrow\infty$, the lemma follows.
\end{proof}

Theorem \ref{cequiv} then follows from combining Lemmas \ref{boost}, \ref{swaplem}, and \ref{lotscanswap} in the appropriate way.

\begin{proof}[Proof of Theorem \ref{cequiv}]
To complete the proof, we show that our construction for equivalences yields $c^nn!$ permutations of $[n]$ with nontrivial equivalences for all $c<1$ and sufficiently large $n$ when combined with Lemma \ref{lotscanswap}.

Let $\epsilon>0$ be a real number such that $c+\epsilon<1$, and let $k$ be a positive integer such that for all sufficiently large $n$, at least $(c+\epsilon)^nn!$ permutations of $[n]$ avoid the consecutive pattern $1\cdots k$ and satisfy the additional property that $1,2,3$ are not consecutive in the permutation and $3$ appears before $1$ and $2$. Consider a permutation $\sigma=\sigma(1)\cdots\sigma(n-k-1)$ of $[n-k-1]$ with the aforementioned properties. We then consider the permutation $\pi=1,\ldots,k,n,\sigma(1)+k,\ldots,\sigma(n-k-1)+k$ of $[n]$. Note that $\pi$ is grounded (its streak is equal to its max height) with streak $k$. This is as $n>\sigma(1)+k$ and $\sigma(1)+k,\ldots,\sigma(n-k-1)+k$ avoids the consecutive pattern $1\cdots k$, so the height of all terms in the permutation are at most $k$. Let $\sigma(p)=1,\sigma(q)=2,\sigma(r)=3$, so $r<p,q$ and no two of $p,q,r$ are consecutive. We then have that $\pi(k+1+p)=k+1$ and $\pi(k+1+q)=k+2$ are not consecutive in $\pi$ and have the same height of $1$. Since $\pi(k)=k$ and $\pi(k+1+r)=k+3$ appear before $k+1$ and $k+2$ in $\pi$, by our construction $\pi$ is c-forest-Wilf equivalent to $\pi'$, where $\pi'=\pi(1),\ldots,\pi(k+p),\pi(k+q+1),\pi(k+p+2),\ldots,\pi(k+q),\pi(k+p+1),\pi(k+q+2),\ldots,\pi(n)$, i.e. $\pi$ but we switch $k+1$ and $k+2$. This is essentially an application of Lemmas \ref{boost} and \ref{swaplem}, but we have written it in this way to count the number of pairs obtained. Since $\pi$ and $\pi'$ are not complements, $\pi$ is a part of a nontrivial c-forest-Wilf equivalence. Each such $\sigma$ results in a different $\pi$, and there are at least $(c+\epsilon)^{n-k-1}(n-k-1)!\ge\frac{(c+\epsilon)^nn!}{((c+\epsilon)n)^{k+1}}$ such $\sigma$. This is at least $c^nn!$ for sufficiently large $n$, as desired.
\end{proof}

\subsection{Necessary conditions for strong equivalences between grounded permutations}
\label{ground}
\hspace*{\fill} \\
Subsection \ref{cfam} establishes some sufficient conditions for grounded permutations to be equivalent, and in this subsection we examine the other direction. Our goal is to determine some necessary conditions for two grounded permutations $\pi$ and $\pi'$ to be equivalent. We follow the same proof methodology of Theorem \ref{gp15} given by Garg and Peng, finding formulas for specific cluster numbers and comparing them between $\pi$ and $\pi'$. Note that if $\pi\stackrel c\sim\pi'$, then $\pi$ and $\pi'$ have the same length.

\begin{prop}
\label{sameh}
If $\pi$ and $\pi'$ are grounded permutations of length $k$ with $\pi\stackrel{sc}\sim\pi'$, then $\pi$ and $\pi'$ have the same streak. Furthermore, the heights of $i$ in $\pi$ and $\pi'$ are equal for all $i\in\{2,\ldots,k\}$.
\end{prop}

\begin{proof}
To proceed, define the forest cluster numbers $r_{n,m}$ for $\pi$ as the number of $m$-clusters of size $n$ for $\pi$. Suppose that $\pi$ has streak $s$. We will consider $r_{2k-h,2}$ for $h\le s+1$. Note that $2$-clusters of size $2k-h$ consist of two instances of $\pi$ that intersect at $h$ vertices. One way in which this occurs is that the two instances share the first $h$ vertices. When $h=s+1$, there are no other ways in which this occurs, since the only increasing consecutive subsequence of $\pi$ of length $s+1$ is at the beginning. For $h\le s$, let $S_h\subseteq\{2,\ldots,k\}$ denote the terms of $\pi$ of height at most $h$ (here we refer to the terms themselves, not the indices). Figure \ref{groundclustfig} shows examples of such clusters.

\begin{figure}[ht]
    \centering
    \forestset{filled circle/.style={
      circle,
      text width=4pt,
      fill,
    },}
    \scalebox{0.7}{
    \begin{minipage}[b]{0.45\linewidth}
    \centering
    \begin{forest}
    for tree={filled circle, inner sep = 0pt, outer sep = 0 pt, s sep = 1 cm}
    [, fill=white
        [, edge=white, name=v1
            [, edge=red, name=v2
                [, edge=white, name=v3
                    [, edge=white, name=v4
                        [, edge=red, name=v5
                            [, edge=red, name=v6
                                [, edge=red, name=v7
                                ]
                            ]
                        ]
                        [, edge=blue, name=v8
                            [, edge=blue, name=v9
                                [, edge=blue, name=v10
                                    [, edge=blue, name=v11
                                    ]
                                ]
                            ]
                        ]
                    ]
                ]
            ]
        ]
    ]
    \path[red,out=-135,in=135] (v2.child anchor) edge (v3.parent anchor);
    \path[red,out=-135,in=135] (v3.child anchor) edge (v4.parent anchor);
    \path[blue,out=-45,in=45] (v2.child anchor) edge (v3.parent anchor);
    \path[blue,out=-45,in=45] (v3.child anchor) edge (v4.parent anchor);
    \end{forest}
    \end{minipage}
    \begin{minipage}[b]{0.45\linewidth}
    \centering
    \begin{forest}
    for tree={filled circle, inner sep = 0pt, outer sep = 0 pt, s sep = 1 cm}
    [, fill=white
        [, edge=white, name=v1
            [, edge=white, name=v2
                [, edge=white, name=v3
                    [, edge=white, name=v4
                        [, edge=white, name=v5
                            [, edge=red, name=v6
                                [, edge=red, name=v7
                                ]
                            ]
                            [, edge=blue, name=v8
                                [, edge=blue, name=v9
                                ]
                            ]
                        ]
                    ]
                ]
            ]
        ]
    ]
    \path[red,out=-135,in=135] (v1.child anchor) edge (v2.parent anchor);
    \path[red,out=-135,in=135] (v2.child anchor) edge (v3.parent anchor);
    \path[red,out=-135,in=135] (v3.child anchor) edge (v4.parent anchor);
    \path[red,out=-135,in=135] (v4.child anchor) edge (v5.parent anchor);
    \path[blue,out=-45,in=45] (v1.child anchor) edge (v2.parent anchor);
    \path[blue,out=-45,in=45] (v2.child anchor) edge (v3.parent anchor);
    \path[blue,out=-45,in=45] (v3.child anchor) edge (v4.parent anchor);
    \path[blue,out=-45,in=45] (v4.child anchor) edge (v5.parent anchor);
    \end{forest}
    \end{minipage}}
    \caption{Clusters for a pattern of length $7$ with two instances.}
    \label{groundclustfig}
\end{figure}
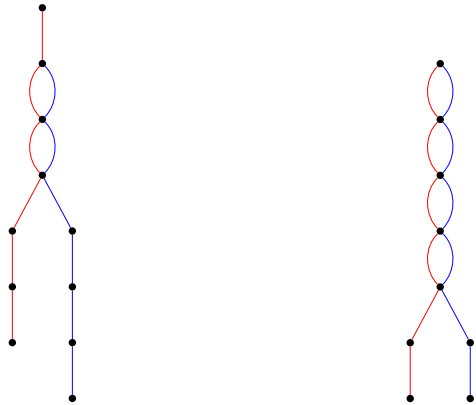

If the two instances of $\pi$ do not share their first vertex, then one begins at the root of the cluster and the other begins at some non-root vertex of the first instance. Let $v$ be the last vertex that the two instances share, and let $i$ be the label of $v$. Note that all labels of the second instance that are descendants of $v$ are at least $i$. It follows that the labels $1,\ldots,i$ all lie in the first instance, so if $v$ is the $j$th vertex in the first instance, then $\pi(j)=i$, and the height of $i$ in $\pi$ is at most $h$. The number of such clusters is then $\binom{2k-h-i}{k-h}$, as it suffices to determine the remaining $k-h$ labels in the second instance, where the possible values are in $\{i+1,\ldots,2k-h-i\}$. There are thus $\sum_{i\in S_h}\binom{2k-h-i}{k-h}$ such clusters.

If the two instances of $\pi$ share their first vertex, then they share their first $h$ vertices. For $h\le s$, the first $h$ labels must be $1,\ldots,h$. Then there are $\frac12\binom{2k-2h}{k-h}$ ways to determine the rest of the labels, as we must decide which $k-h$ labels are in the first instance and then divide by two since the order of the instances does not matter when they share their first vertex. When $h=s+1$, there are $\frac12\binom{2(\pi(s+1)-s-1)}{\pi(s+1)-s-1}\binom{2(k-\pi(s+1))}{k-\pi(s+1)}$ such clusters. Indeed, the labels of the first $s$ vertices are $1,\ldots,s$, and the label of the $(s+1)$st vertex $v$ must be $s+2(\pi(s+1)-s-1)+1$, since among the remaining vertices exactly $2(\pi(s+1)-s-1)$ have labels less than $v$. It the remains to decide which $\pi(s+1)-s-1$ of the elements of $\{s+1,\ldots,s+2(\pi(s+1)-s-1)\}$ and which $k-\pi(s+1)$ of the elements of $\{s+2(\pi(s+1)-s-1)+2,\ldots,2k-s-1\}$ are labels in the first instance, and then divide by $2$ due to the indistinguishability of the two instances.

Thus, $r_{2k-h,2}=\frac12\binom{2k-2h}{k-h}+\sum_{i\in S_h}\binom{2k-h-i}{k-h}$ for $h\le s$, and for $h=s+1$ we have that \[r_{2k-s-1,2}=\frac12\binom{2(\pi(i+1)-s-1)}{\pi(s+1)-s-1}\binom{2(k-\pi(s+1))}{k-\pi(s+1)}.\] Define the cluster numbers $r_{n,m}'$ and sets $S_h'$ for $\pi'$ analogously, and note that by Theorem 3.2, if $\pi$ and $\pi'$ are strong-c-forest-Wilf equivalent then $r_{n,m}=r_{n,m}'$ for all $m,n$. Suppose that the streak $s'$ of $\pi'$ is greater than $s$. We have that \[\frac12\binom{2(\pi(i+1)-s-1)}{\pi(s+1)-s-1}\binom{2(k-\pi(s+1))}{k-\pi(s+1)}=\frac12\binom{2k-2s-2}{k-s-1}+\sum_{i\in S_{s+1}'}\binom{2k-s-1-i}{k-s-1}\] since $r_{2k-s-1,2}=r_{2k-s-1,2}'$. But note that the right-hand side is strictly greater than the left-hand side. Indeed, the sequence $\binom{2a}a\binom{2n-2a}{n-a}$ is strictly decreasing for $a\le\frac n2$ as \[\frac{\binom{2a+2}{a+1}\binom{2n-2a-2}{n-a-1}}{\binom{2a}a\binom{2n-2a}{n-a}}=\frac{(2a+1)(n-a)}{(a+1)(2n-2a-1)}<1\] for $n>2a+1$. Then as $S_{i+1}'$ is nonempty, the sum on the right-hand side is positive while the binomial on the right-hand side is at least the binomial on the left-hand side. Hence, it cannot happen that $r_{2k-s-1,2}=r_{2k-s-1,2}'$ if $s\ne s'$, so $s=s'$, as desired.

We then consider the equalities $r_{2k-h,2}=r_{2k-h,2}'$ for all $h\le s$. They tell us that \[\frac12\binom{2k-2h}{k-h}+\sum_{i\in S_h}\binom{2k-h-i}{k-h}=\frac12\binom{2k-2h}{k-h}+\sum_{i\in S_h'}\binom{2k-h-i}{k-h},\] so $\sum_{i\in S_h}\binom{2k-h-i}{k-h}=\sum_{i\in S_h'}\binom{2k-h-i}{k-h}$. We claim that this means that $S_h=S_h'$. Indeed, the elements of $S_h$ and $S_h'$ are greater than $h$, and for $i=h+1,\ldots,k$ the value of $\binom{2k-h-i}{k-h}$ is among $\binom{2k-2h-1}{k-h},\ldots,\binom{k-h}{k-h}$. But each element in this sequence is more than twice the next element, as $\frac{\binom{k-h+i}{k-h}}{\binom{k-h+i-1}{k-h}}=\frac{k-h+i}i>2$ for $i<k-h$. It follows that $\sum_{i\in S}\binom{2k-h-i}{k-h}$ takes on different values for different $S\subseteq\{h+1,\ldots,k\}$, which means that $\sum_{i\in S_h}\binom{2k-h-i}{k-h}=\sum_{i\in S_h'}\binom{2k-h-i}{k-h}$ can only occur when $S_h=S_h'$. The sets of terms of height at most $h$ are equal for $\pi$ and $\pi'$ for all $h\le s$. It follows that the heights of all $i\in\{2,\ldots,k\}$ with respect to $\pi$ and $\pi'$ are equal, as desired.
\end{proof}

\begin{prop}
\label{weakcons}
If $\pi$ and $\pi'$ are grounded permutations of length $k$ that are strong-c-forest-Wilf equivalent with streak $s$, then $\pi(s+1)=\pi'(s+1)$ or $\pi(s+1)+\pi'(s+1)=s+k+1$. Furthermore, if $d$ is the least positive integer such that $\{\pi(d),\ldots,\pi(k)\}$ does not contain two consecutive numbers and if $d'$ is the least positive integer such that $\{\pi'(d'),\ldots,\pi'(k)\}$ does not contain two consecutive numbers, then $d=d'$.
\end{prop}

\begin{proof}
We will prove this proposition by considering the equation $r_{2k-h,2}=r_{2k-h,2}'$ for $h>s$. Note that in a $2$-cluster of size $2k-h$ for $h>s$, the two instances must share the first $h$ vertices. By the same arguments as in Proposition \ref{sameh}, if $\{s+1,\ldots,k\}\setminus\{\pi(s+1),\ldots,\pi(h)\}$ has consecutive blocks of size $a_1,\ldots,a_i$, then $r_{2k-h,2}=\frac12\binom{2a_1}{a_1}\cdots\binom{2a_i}{a_i}$. By taking $h=s+1$, we have that $a_1=\pi(i+1)-s-1$ and $a_2=k-\pi(i+1)$. If we define $a_1',\ldots,a_i'$ analogously for $\pi'$, we have that $a_1'=\pi'(i+1)-s-1$ and $a_2'=k-\pi'(i+1)$. As before, the sequence $\binom{2a}a\binom{2n-2a}{n-a}$ is strictly decreasing for $a\le\frac n2$, so $\frac12\binom{2a_1}{a_1}\binom{2a_2}{a_2}=\frac12\binom{2a_1'}{a_1'}\binom{2a_2'}{a_2'}$ and thus $\{a_1,a_2\}=\{a_1',a_2'\}$ since $a_1+a_2=a_1'+a_2'$. This implies that $\pi(s+1)=\pi'(s+1)$ or $\pi(s+1)+\pi'(s+1)=s+k+1$. Note that $r_{2k-h,2}=2^{k-h-1}$ for all $h\ge d$, while $r_{2k-d+1,2}=3\cdot2^{k-d-1}$. The same holds true for $r'$ and $d'$, which implies the second part of the proposition.
\end{proof}

\subsection{The triviality of super-strong equivalences}
\label{sstriv}
\hspace*{\fill} \\
In this subsection, we prove Theorem \ref{fssequ}. We follow the same methodology of counting linear extensions of cluster posets as in \cite{DE, LS}. We again consider two permutations $\pi=\pi(1)\cdots\pi(k)$ and $\pi'=\pi'(1)\cdots\pi'(k)$ of length $k$, assumed to satisfy $\pi\stackrel{ssc}\sim\pi'$. Let $m=\left\lfloor\frac k2\right\rfloor$.

For an unlabeled rooted forest $F$ on $n$ vertices with a set $S$ of highlighted paths of length $k$, let $a_{F,S}$ be the number of ways to label the vertices of $F$ with distinct labels in $[n]$ such that the instances of $\pi$ are precisely the highlighted paths in $S$. If $a_{F,S}'$ is defined analogously for $\pi'$, note that by definition, $\pi\stackrel{ssc}\sim\pi'$ if and only if $a_{F,S}=a_{F,S}'$ for all $F$ and $S$.

\begin{defn}
\label{unlabdef}
An \emph{unlabeled cluster} of order $k$ is a rooted tree with a set of highlighted paths of length $k$ such that every vertex belongs to a highlighted path and it is not possible to partition the vertex set into two parts such that no path contains vertices in both parts. In other words, an unlabeled cluster is simply a forest cluster without the labels. For each unlabeled cluster $C$ of order $k$, define $r_C$ and $r_C'$ to be the number of ways to label the vertices in $C$ that result in a forest cluster for $\pi$ and $\pi'$, respectively.
\end{defn}

\begin{defn}
\label{fcpdef}
Given an unlabeled cluster $C$ of order $k$, the \emph{forest cluster poset} $P_C$ is the poset on the vertex set of $C$ generated by the following relations: for each highlighted path $v_1,\ldots,v_k$, $v_{\pi^{-1}(1)}<\cdots<v_{\pi^{-1}(k)}$, where $\pi^{-1}$ is the inverse of $\pi$ in $S_k$. We define $P_C'$ analogously for $\pi'$.
\end{defn}

Note that $r_C$ and $r_C'$ are respectively the number of linear extensions of $P_C$ and $P_C'$ divided by the number of automorphisms of the underlying rooted tree of $C$.

We first prove the following relation between the refined forest cluster numbers $r_C$ and super-strong equivalence for forests. Our proof mimics the analogous result of Dwyer and Elizalde for permutations given in \cite{DE}.

\begin{lem}
\label{rfcnlem}
We have that $\pi\stackrel{ssc}\sim\pi'$ if and only if $r_C=r_C'$ for all unlabeled clusters $C$ of order $k$.
\end{lem}

\begin{proof}
Let $b_{F,S}$ be the number of ways to label $F$ such that all paths in $S$ are instances of $\pi$, though not necessarily all instances of $\pi$ are in $S$, and define $b_{F,S}'$ analogously for $\pi'$. We have that $b_{F,S}=\sum_{S\subseteq T}a_{F,T}$ and $b_{F,S}'=\sum_{S\subseteq T}a_{F,T}'$. By the Principle of Inclusion-Exclusion, $a_{F,S}=a_{F,S}'$ for all $F$ and $S$ if and only if $b_{F,S}=b_{F,S}'$ for all $F$ and $S$, so $\pi\stackrel{ssc}\sim\pi'$ if and only if $b_{F,S}=b_{F,S}'$ for all $F$ and $S$.

It is clear that the equality $b_{F,S}=b_{F,S}'$ holding for all $F$ and $S$ implies $r_C=r_C'$ for all $C$, as $C$ is an unlabeled cluster where the underlying forest is a tree. In the other direction, consider the finest partition of the vertex set of $F$ such that any two adjacent vertices in a highlighted path are in the same part. In this way, we have partitioned the vertices into singleton sets and trees $T_1,\ldots,T_p$ with sets $S_1,\ldots,S_p$ of highlighted paths such that $T_i$ and $S_i$ form an unlabeled cluster $C_i$. If the number of vertices of $T_i$ is $n_i$ and the number of vertices of $F$ is $n$, then in a uniform random labelling of $F$, the probability that all highlighted paths are instances of $\pi$ is $\frac{b_{F,S}}{n!}$. On the other hand, it is also equal to $\prod_{i=1}^p\frac{r_{C_i}}{n_i!}$ since the events $E_i$ that the paths in $S_i$ are all instances of $\pi$ are independent. We thus have that $\frac{b_{F,S}}{n!}=\prod_{i=1}^p\frac{r_{C_i}}{n_i!}$ and analogously that $\frac{b_{F,S}'}{n!}=\prod_{i=1}^p\frac{r_{C_i}'}{n_i!}$, which proves the reverse direction, establishing the lemma.
\end{proof}

Recall that by Theorem \ref{gp15} and complementation, we may assume that $\pi(1)=\pi'(1)$ since $\pi\stackrel{ssc}\sim\pi'$. We first construct a certain family of clusters in Lemma \ref{sselem1} that prove that under this assumption, in most cases the two permutations must be the same, which would imply the theorem. The one caveat is that if $\pi(1)=\frac{k-1}2$, we are unable to distinguish between complements with only Lemma \ref{sselem1}, so Lemma \ref{sselem2} is also required to prove the result. In contrast to \cite{EN,LS}, our approach is not to show a contradiction of asymptotics but rather to use the polynomial identity theorem: two polynomials that agree on an infinite set must be equal. The extra flexibility of forests allows us to construct families of cluster posets whose linear extension counts are effectively given by polynomials. The key innovation in our proof is to rely on the algebraic properties of polynomials, such as their roots or the ability to cancel out common factors, that may not have an obvious combinatorial interpretation, instead of on how fast the linear extension counts grow.

\begin{lem}
\label{sselem1}
Suppose that $\pi(1)=\pi'(1)$ and $\pi\stackrel{ssc}\sim\pi'$. If it is not the case that $\pi(1)=m$ and $k=2m-1$, then $\pi=\pi'$. If $\pi(1)=m$ with $k=2m-1$, then $\pi(i)=\pi'(i)$ or $\pi(i)+\pi'(i)=k+1$ for all $i$.
\end{lem}

\begin{proof}
Fix an $i>1$ and let $q=\pi(1)$. Consider the unlabeled cluster $C_n$ of order $k$ defined as follows: we have one instance beginning at the root of $C_n$, and there are $n$ disjoint instances starting from the $i$th vertex of this instance. Thus, the intersection of any pair of instances consists only of the $i$th vertex of the first instance. Suppose that $\pi(i)=p$ and $\pi'(i)=p'$. We now describe the cluster poset $P_{C_n}$. It consists of $n+1$ chains with one common element $v$. In one of these chains, $v$ is the $p$th smallest element, and in the rest of the chains, $v$ is the $q$th smallest element. This poset has $\binom{(q-1)n+p-1}{q-1,\ldots,q-1,p-1}\binom{(k-q)n+k-p}{k-q,\ldots,k-q,k-p}$ linear extensions, as the label of $v$ is determined and the poset less than $v$ and the poset greater than $v$ all consist of $n$ disjoint chains. Similarly, $P_{C_n}'$ has $\binom{(q-1)n+p'-1}{q-1,\ldots,q-1,p'-1}\binom{(k-q)n+k-p'}{k-q,\ldots,k-q,k-p'}$ linear extensions. Figure \ref{ssefig1} gives an example of an unlabeled cluster $C_n$ and its cluster poset $P_{C_n}$.

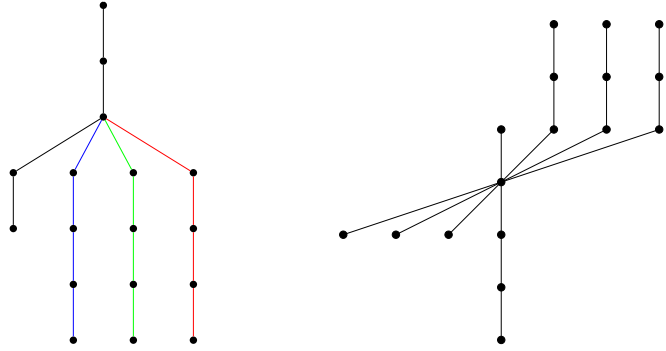
\begin{figure}[ht]
    \centering
    \forestset{filled circle/.style={
      circle,
      text width=4pt,
      fill,
    },}
    \scalebox{0.7}{
    \begin{minipage}[b]{0.45\linewidth}
    \centering
    \begin{forest}
    for tree={filled circle, inner sep = 0pt, outer sep = 0 pt, s sep = 1 cm}
    [, fill=white
        [, edge=white
            [,
                [,
                    [,
                        [,
                        ]
                    ]
                    [, edge=blue
                        [, edge=blue
                            [, edge=blue
                                [, edge=blue
                                ]
                            ]
                        ]
                    ]
                    [, edge=green
                        [, edge=green
                            [, edge=green
                                [, edge=green
                                ]
                            ]
                        ]
                    ]
                    [, edge=red
                        [, edge=red
                            [, edge=red
                                [, edge=red
                                ]
                            ]
                        ]
                    ]
                ]
            ]
        ]
    ]
    \end{forest}
    \end{minipage}
    \begin{minipage}[b]{0.45\linewidth}
    \centering
    \begin{tikzpicture}
    \draw[fill] (0,0) circle(2pt);
    \draw[fill] (0,1) circle(2pt);
    \draw[fill] (0,2) circle(2pt);
    \draw[fill] (0,3) circle(2pt);
    \draw[fill] (0,4) circle(2pt);
    \draw[fill] (-3,2) circle(2pt);
    \draw[fill] (-2,2) circle(2pt);
    \draw[fill] (-1,2) circle(2pt);
    \draw[fill] (1,4) circle(2pt);
    \draw[fill] (2,4) circle(2pt);
    \draw[fill] (3,4) circle(2pt);
    \draw[fill] (1,5) circle(2pt);
    \draw[fill] (2,5) circle(2pt);
    \draw[fill] (3,5) circle(2pt);
    \draw[fill] (1,6) circle(2pt);
    \draw[fill] (2,6) circle(2pt);
    \draw[fill] (3,6) circle(2pt);
    \draw (0,0)--(0,1)--(0,2)--(0,3)--(0,4);
    \draw (-1,2)--(0,3)--(1,4)--(1,5)--(1,6);
    \draw (-2,2)--(0,3)--(2,4)--(2,5)--(2,6);
    \draw (-3,2)--(0,3)--(3,4)--(3,5)--(3,6);
    \end{tikzpicture}
    \end{minipage}}
    \caption{The unlabeled cluster $C_n$ for a pattern $\pi$ of length $5$ and the Hasse diagram for the cluster poset $P_{C_n}$ when $n=3$, $i=3$, $\pi(1)=2$, and $\pi(3)=4$.}
    \label{ssefig1}
\end{figure}

If $\pi$ and $\pi'$ are super-strongly equivalent, then $P_{C_n}$ and $P_{C_n}'$ have the same number of linear extensions so $\binom{(q-1)n+p-1}{q-1,\ldots,q-1,p-1}\binom{(k-q)n+k-p}{k-q,\ldots,k-q,k-p}=\binom{(q-1)n+p'-1}{q-1,\ldots,q-1,p'-1}\binom{(k-q)n+k-p'}{k-q,\ldots,k-q,k-p'}$ for all $n$. This is equivalent to \[\binom{(q-1)n+p-1}{p-1}\binom{(k-q)n+k-p}{k-p}=\binom{(q-1)n+p'-1}{p'-1}\binom{(k-q)n+k-p'}{k-p'}\] holding for all $n$. This holds as a polynomial identity for all $n$, so it follows that \[\binom{(q-1)x+p-1}{p-1}\binom{(k-q)x+k-p}{k-p}=\binom{(q-1)x+p'-1}{p'-1}\binom{(k-q)x+k-p'}{k-p'}\] as polynomials. The roots of the left-hand side and right-hand side then must be the same. It is easy to see that the roots of the left-hand side are $-\frac1{q-1},\ldots,-\frac{p-1}{q-1},-\frac1{k-q},\ldots,-\frac{k-p}{k-q}$ and the roots of the right-hand side are $-\frac1{q-1},\ldots,-\frac{p'-1}{q-1},-\frac1{k-q},\ldots,-\frac{k-p'}{k-q}$. If $q-1\ne k-q$, then we must have that $p=p'$, corresponding to when it is not the case that $\pi(1)=m$ and $k=2m-1$. Otherwise, we must have that $p=p'$ or $p+p'=k+1$, so the lemma follows.
\end{proof}

Note that when $\pi(1)=\frac{k+1}2$, the family of cluster posets we construct in the proof of Lemma \ref{sselem1} cannot distinguish between the two cases $\pi(i)=\pi'(i)$ or $\pi(i)+\pi'(i)=k+1$ because if we were to construct the same family for $\overline\pi$, then we would get an isomorphism between the corresponding cluster posets after reversing the relative order. Any proof of Theorem \ref{fssequ} must take into account the case that $\pi'=\overline\pi$. When $\pi(1)\ne\frac{k+1}2$, this is taken care of by the complementation we use at the beginning to assume $\pi(1)=\pi'(1)$, but this does not quite work when $\pi(1)=\frac{k+1}2$. Lemma \ref{sselem1} shows that each individual term $\pi'(i)$ of $\pi'$ is as we expect (equal to $\pi(i)$ or $k+1-\pi(i)$) but cannot show that the choice between $\pi(i)$ and $k+1-\pi(i)$ is uniform over all $i$. To fix this, we add another ``branch'' to the family of clusters we constructed to introduce asymmetry between $\pi$ and $\overline\pi$. By Lemma \ref{sselem1}, we can take complements to fix at least one other term of the pattern between $\pi$ and $\pi'$. We make a specific choice for the term fixed and modify our family of clusters to show that when $\pi\ne\pi'$ the cluster posets do not have the same number of linear extensions for every cluster. The proof of Lemma \ref{sselem2} uses a few somewhat tricky algebraic manipulations. It would be interesting to see a more combinatorial proof for Lemma \ref{sselem1} and especially Lemma \ref{sselem2}.

\begin{lem}
\label{sselem2}
If $\pi(1)=\pi'(1)=m$ and $\pi\stackrel{ssc}\sim\pi'$ with $k=2m-1$ and $\pi(\ell)=\pi'(\ell)$ for some $\ell>1$, then $\pi=\pi'$.
\end{lem}

\begin{proof}
Let $i$ be such that $\pi(i)=1$. By Lemma \ref{sselem1}, we know that $\pi'(i)=1$ or $\pi'(i)=k$. By complementation, we may without loss of generality suppose that $\pi'(i)=1$. Suppose that $\pi(j)=a$. We will show that it cannot be the case that $\pi'(j)=k+1-a$, so $\pi'(j)=a$. Then it will follow that $\pi$ and $\pi'$ must be either equal or complements, so if $\pi$ and $\pi'$ agree at two indices then they are equal. We already know this to be the case for $\pi(j)=1,m,k$ by assumption, so we only need to consider $a\ne1,m,k$.

Consider the unlabeled cluster $C_n$ of order $k$ defined as follows: we have one instance $I$ beginning at the root of $C$, $n$ instances $J_1,\ldots,J_n$ beginning at the $i$th vertex of $I$ and sharing no other vertices with $I$ or each other, and one instance $K$ beginning at the $j$th vertex of $I$ and sharing no other vertices with $I,J_1,\ldots,J_n$. We now describe the cluster posets $P_{C_n}$ and $P_{C_n}'$, assuming that $\pi'(j)=k+1-a$. Both cluster posets consist of $n+2$ chains of length $k$, corresponding to the $n+2$ instances in $C$. Let the \emph{central chain} be the chain corresponding to $I$ with lowest element $u$, the \emph{lower chains} be the chains corresponding to $J_1,\ldots,J_n$, and the \emph{upper chain} be the chain corresponding to $K$, which intersects the central chain at the element $v$. Note that $u$ is the common median and only pairwise common element of the lower chains and that $v$ is the median of the upper chain as $\pi(1)=m$. In $P_{C_n}$, $v$ is the $a$th smallest element of the central chain and in $P_{C_n}'$, $v$ is the $a$th largest element of the central chain. Figure \ref{ssefig2} gives an example of an unlabeled cluster $C_n$ and its cluster poset $P_{C_n}$.

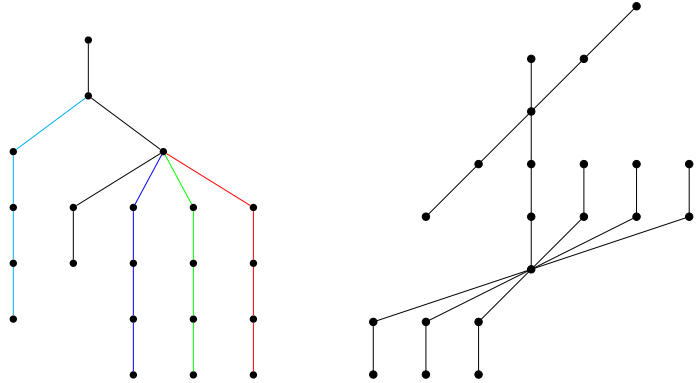
\begin{figure}[ht]
    \centering
    \forestset{filled circle/.style={
      circle,
      text width=4pt,
      fill,
    },}
    \scalebox{0.7}{
    \begin{minipage}[b]{0.45\linewidth}
    \centering
    \begin{forest}
    for tree={filled circle, inner sep = 0pt, outer sep = 0 pt, s sep = 1 cm}
    [, fill=white
        [, edge=white
            [,
                [, edge=cyan
                    [, edge=cyan
                        [, edge=cyan
                            [, edge=cyan
                            ]
                        ]
                    ]
                ]
                [,
                    [,
                        [,
                        ]
                    ]
                    [, edge=blue
                        [, edge=blue
                            [, edge=blue
                                [, edge=blue
                                ]
                            ]
                        ]
                    ]
                    [, edge=green
                        [, edge=green
                            [, edge=green
                                [, edge=green
                                ]
                            ]
                        ]
                    ]
                    [, edge=red
                        [, edge=red
                            [, edge=red
                                [, edge=red
                                ]
                            ]
                        ]
                    ]
                ]
            ]
        ]
    ]
    \end{forest}
    \end{minipage}
    \begin{minipage}[b]{0.45\linewidth}
    \centering
    \begin{tikzpicture}
    \draw[fill] (0,0) circle(2pt);
    \draw[fill] (0,1) circle(2pt);
    \draw[fill] (0,2) circle(2pt);
    \draw[fill] (0,3) circle(2pt);
    \draw[fill] (0,4) circle(2pt);
    \draw[fill] (-1,-1) circle(2pt);
    \draw[fill] (-2,-1) circle(2pt);
    \draw[fill] (-3,-1) circle(2pt);
    \draw[fill] (-1,-2) circle(2pt);
    \draw[fill] (-2,-2) circle(2pt);
    \draw[fill] (-3,-2) circle(2pt);
    \draw[fill] (1,1) circle(2pt);
    \draw[fill] (2,1) circle(2pt);
    \draw[fill] (3,1) circle(2pt);
    \draw[fill] (1,2) circle(2pt);
    \draw[fill] (2,2) circle(2pt);
    \draw[fill] (3,2) circle(2pt);
    \draw[fill] (-2,1) circle(2pt);
    \draw[fill] (-1,2) circle(2pt);
    \draw[fill] (1,4) circle(2pt);
    \draw[fill] (2,5) circle(2pt);
    \draw (0,0)--(0,1)--(0,2)--(0,3)--(0,4);
    \draw (-1,-2)--(-1,-1)--(0,0)--(1,1)--(1,2);
    \draw (-2,-2)--(-2,-1)--(0,0)--(2,1)--(2,2);
    \draw (-3,-2)--(-3,-1)--(0,0)--(3,1)--(3,2);
    \draw (-2,1)--(-1,2)--(0,3)--(1,4)--(2,5);
    \end{tikzpicture}
    \end{minipage}}    
    \caption{The unlabeled cluster $C_n$ for a pattern $\pi$ of length $5$ and the Hasse diagram for the cluster poset $P_{C_n}$ when $n=3$, $i=3$, $j=2$, and $a=4$.}
    \label{ssefig2}
\end{figure}

Now, we will enumerate linear extensions of $P_{C_n}$ and $P_{C_n}'$. We proceed using casework on the number $s$ of elements of the upper chain with label less than the label of $v$. Here, we must have $s<m$ as there are only $m-1$ elements of the upper chain that are incomparable to $v$ and all other elements are greater than $v$. Note that the label of $u$ must be $(m-1)n+s+1$. There are then $\binom{(m-1)n+s}s$ ways to choose which of the numbers in $[(m-1)n+s]$ are labels of elements in the upper chain. We next choose the labels of the remaining elements not in one of the lower chains. Note that there are $4(m-1)-s$ such elements, and the labels can be anything in $\{(m-1)n+s+2,\ldots,2(m-1)n+4(m-1)+1\}$, resulting in $\binom{(m-1)n+4(m-1)-s}{4(m-1)-s}$ ways. We then assign the specific labels to the elements not in the lower chains. The $s$ smallest elements of the upper chain already have their labels determined. The remaining elements form a poset consisting of two chains that intersect at $v$. One of these chains has $m-1$ elements greater than $v$ and $m-s-1$ elements less than $v$. For $P_{C_n}$ the other chain has $2m-a-1$ elements greater than $v$ and $a-2$ elements less than $v$, and for $P_{C_n}'$ the other chain has $a-1$ elements greater than $v$ and $2m-a-2$ elements less than $v$. Thus, there are $\binom{3m-a-2}{m-1}\binom{m-s+a-3}{m-s-1}$ ways to do so for $P_{C_n}$ and $\binom{m+a-2}{m-1}\binom{3m-s-a-3}{m-s-1}$ ways to do so for $P_{C_n}'$. Finally, to decide the labels of the elements in the lower chains, note that there are $N=\binom{(m-1)n}{m-1,\ldots,m-1}$ ways to choose the labels of the elements smaller than $u$ as those just consist of $n$ disjoint chains of $m-1$ elements. Similarly, there are $N$ ways to choose the labels of the elements larger than $u$. Thus, the number of linear extensions of $P_{C_n}$ and $P_{C_n}'$ are \[N^2\sum_{s=0}^{m-1}\binom{(m-1)n+s}s\binom{(m-1)n+4(m-1)-s}{4(m-1)-s}\binom{3m-a-2}{m-1}\binom{m-s+a-3}{m-s-1}\] and \[N^2\sum_{s=0}^{m-1}\binom{(m-1)n+s}s\binom{(m-1)n+4(m-1)-s}{4(m-1)-s}\binom{m+a-2}{m-1}\binom{3m-s-a-3}{m-s-1},\] respectively.

To have $\pi$ super-strongly equivalent to $\pi$, these two quantities must be equal for all $n$. Dividing both sides by $N^2$, we have by Lemma \ref{rfcnlem} that \[\sum_{s=0}^{m-1}\binom{(m-1)n+s}s\binom{(m-1)n+4(m-1)-s}{4(m-1)-s}\binom{3m-a-2}{m-1}\binom{m-s+a-3}{m-s-1}\]\[=\sum_{s=0}^{m-1}\binom{(m-1)n+s}s\binom{(m-1)n+4(m-1)-s}{4(m-1)-s}\binom{m+a-2}{m-1}\binom{3m-s-a-3}{m-s-1}\] for all $n$. Note that both sides are in fact polynomials in $n$. Thus, it follows by substituting $x=(m-1)n$ that \[\sum_{s=0}^{m-1}\binom{x+s}s\binom{x+4(m-1)-s}{4(m-1)-s}\binom{3m-a-2}{m-1}\binom{m-s+a-3}{m-s-1}\]\[=\sum_{s=0}^{m-1}\binom{x+s}s\binom{x+4(m-1)-s}{4(m-1)-s}\binom{m+a-2}{m-1}\binom{3m-s-a-3}{m-s-1}\] is a polynomial identity in $x$. Dividing both sides by $x+1$ gives \[\sum_{s=0}^{m-1}\binom{x+s}s\frac{\binom{x+4(m-1)-s}{4(m-1)-s}}{x+1}\binom{3m-a-2}{m-1}\binom{m-s+a-3}{m-s-1}\]\[=\sum_{s=0}^{m-1}\binom{x+s}s\frac{\binom{x+4(m-1)-s}{4(m-1)-s}}{x+1}\binom{m+a-2}{m-1}\binom{3m-s-a-3}{m-s-1}\] where $\frac1{x+1}\binom{x+4(m-1)-s}{4(m-1)-s}=\frac1{4(m-1)-s}\binom{x+4(m-1)-s}{4(m-1)-s-1}$ is a polynomial that does not vanish at $x=-1$. Plugging in $x=-1$, we note that $\binom{s-1}s=0$ for all $s>0$ while $\binom{-1}0=1$. Hence, all terms from $s>0$ will vanish, leaving the identity \[\frac1{4(m-1)}\binom{3m-a-2}{m-1}\binom{m+a-3}{m-1}=\frac1{4(m-1)}\binom{m+a-2}{m-1}\binom{3m-a-3}{m-1}.\] This is equivalent to $\frac{3m-a-2}{2m-a-1}=\frac{m+a-2}{a-1}$, which cannot happen unless $a=m$, a contradiction. Thus, we cannot have that $\pi'(j)=k+1-a$, and the lemma is proven.
\end{proof}

The combination of Lemmas \ref{sselem1} and \ref{sselem2} proves Theorem \ref{fssequ}. By examining our proof closely, we also have the following corollary.

\begin{cor}
\label{clusbij}
If patterns $\pi$ and $\pi'$ satisfy $\pi\stackrel{sc}\sim\pi'$ nontrivially, then by Theorem \ref{gpfclust}, for all $m$ and $n$ the number of $m$-clusters of size $n$ with respect to $\pi$ is equal to the number of $m$-clusters of size $n$ with respect to $\pi'$. No bijection between clusters with respect to $\pi$ and clusters with respect to $\pi'$ can be structure-preserving. By this, we mean that it cannot be the case that such a bijection always maps a cluster $C$ with respect to $\pi$ to a cluster $C'$ with respect to $\pi'$ such that the underlying unlabeled rooted trees of $C$ and $C'$ are isomorphic.
\end{cor}

\begin{proof}
Note that in our proofs of Lemmas \ref{sselem1} and \ref{sselem2}, the families of unlabeled clusters we construct have the following property: given the underlying rooted tree structure $T$ of the unlabeled cluster and the number $m$ of highlighted instances, there is only one way to highlight $m$ paths of length $k$ in $T$ to yield an unlabeled cluster. This is as $T$ has exactly $m$ leaves, each of which must correspond to the endpoint of a highlighted instance. If $\pi\stackrel{sc}\sim\pi'$, then we have that $\pi(1)=\pi'(1)$ or $\pi(1)+\pi'(1)=k+1$ by Theorem \ref{gp15}. In that case, the proofs of Lemmas \ref{sselem1} and \ref{sselem2} tell us that the number of linear extensions of the corresponding cluster posets, which is precisely the number of clusters for $\pi$ or $\pi'$ with underlying rooted tree $T$ up to a factor of automorphisms, cannot be equal unless $\pi=\pi'$ or $\pi'=\overline\pi$.
\end{proof}

\section*{Acknowledgments} This research was funded by NSF-DMS grant 1949884 and NSA grant H98230-20-1-0009. The author thanks Amanda Burcroff, Swapnil Garg, and Alan Peng for fruitful discussions about this research and for reading drafts of the paper, as well as Noah Kravitz, Ashwin Sah, and Fan Zhou for helpful suggestions. The author also thanks Professor Joe Gallian for suggesting the problem and running the Duluth REU in which this research was conducted and Amanda Burcroff, Colin Defant, and Yelena Mandelshtam for fostering a productive virtual research environment through their mentorship. Finally, the author thanks the referee for catching errors through a careful reading.

\end{document}